\documentclass{amsart}
\usepackage{amsmath}
\usepackage{amssymb}
\usepackage{enumerate}
\usepackage{mathrsfs}
\usepackage{cases}
\usepackage{color}
\usepackage{tikz}
\usepackage{appendix}
\usepackage{makecell}
\usepackage{multirow}
\usepackage{algpseudocode}
\usepackage{algorithmicx,algorithm}
\usepackage{graphicx}
\definecolor{darkblue}{rgb}{0.0, 0.0, 0.55}
\definecolor{bordeaux}{rgb}{0.34, 0.01, 0.1}
\usepackage[pagebackref,colorlinks,linkcolor=bordeaux,citecolor=darkblue,urlcolor=black,hypertexnames=true]{hyperref}
\usetikzlibrary{arrows,positioning}

\newtheorem{theorem}{Theorem}[section]

\newtheorem{example}[theorem]{Example}

\newtheorem{proposition}[theorem]{Proposition}

\newtheorem{remark}[theorem]{Remark}

\def\R{{\mathbb{R}}}

\def\N{{\mathbb{N}}}

\def\x{{\mathbf{x}}}

\def\y{{\mathbf{y}}}
\def\s{{\mathbf{s}}}

\def\be{{\mathbf{e}}}

\def\a{{\boldsymbol{\alpha}}}

\def\b{{\boldsymbol{\beta}}}
\def\g{{\boldsymbol{\gamma}}}

\def\A{{\mathscr{A}}}
\def\B{{\mathscr{B}}}

\def\supp{\hbox{\rm{supp}}}

\def\Tr{\hbox{\rm{Tr}}}

\def\int{\hbox{\rm{int}}}
\def\New{\hbox{\rm{New}}}
\def\Conv{\hbox{\rm{conv}}}

\def\SE{\hbox{\rm{SE}}}

\def\SDSOS{\hbox{\rm{SDSOS}}}

\newcommand{\vge}{\mathbin{\rotatebox[origin=c]{90}{$\ge$}}}
\newif\ifcomment
\commentfalse
\commenttrue
\usepackage{todonotes}

\newcommand{\newtssos}[1]{{{\color{black}#1}}}
\begin{document}

\title[Chordal-TSSOS: a moment-SOS hierarchy that exploits term sparsity]{Chordal-TSSOS: a moment-SOS hierarchy that exploits term sparsity with chordal extension}
\author{Jie Wang \and Victor Magron \and Jean-Bernard Lasserre}
\date{\today}

\begin{abstract}
\newtssos{This work is a follow-up and a complement to \cite{wang2}} for solving polynomial optimization problems (POPs). The chordal-TSSOS hierarchy that we propose is a new sparse moment-SOS framework based on \emph{term-sparsity} and \emph{chordal extension}. By exploiting term-sparsity of the input polynomials we obtain a two-level hierarchy of semidefinite programming relaxations. The novelty and distinguishing feature of
such relaxations is to obtain quasi block-diagonal matrices obtained in an iterative procedure that performs chordal extension of certain adjacency graphs. The graphs are related to the terms arising in the original data and {\em not} to the links between variables. Various numerical examples demonstrate the efficiency and the scalability of this new hierarchy for both unconstrained and constrained POPs.
\newtssos{The two hierarchies are complementary. While the former TSSOS \cite{wang2} has a theoretical convergence guarantee, the chordal-TSSOS has superior performance but lacks this theoretical guarantee.}
\end{abstract}

\maketitle

\section{Introduction}
Consider the polynomial optimization problem (POP):
\begin{equation*}
(\textrm{Q}):\quad f^*=\inf_{\x}\,\{\,f(\x) : \x\in\mathbf{K}\,\},
\end{equation*}
where $f(\x)\in\R[\x]=\R[x_1,\ldots,x_n]$ is a polynomial and $\mathbf{K}\subseteq\R^{n}$ is the basic semialgebraic set
\begin{equation*}
\mathbf{K} := \{\x\in\R^{n} : g_j(\x)\ge 0, j = 1,\ldots,m\},
\end{equation*}
for some polynomials $g_j(\x)\in\R[\x], j = 1,\ldots,m.$ 
The so-called moment-SOS hierarchy \cite{las1} (where SOS stands for \emph{sum of squares})  
is a powerful approach based on certain specific positivity certificates of real algebraic geometry.
It results in solving a hierarchy of semidefinite program (SDP) relaxations of (Q) whose associated monotone sequence of optimal values converges to $f^*$ from below; in fact the convergence is even finite \emph{generically} \cite{nie2014optimality}. However, in view of the present status of SDP solvers, the moment-SOS hierarchy does not scale well and is so far limited to problems of modest size. 

To address the issue of scalability, an important research direction is to define alternative relaxations of (Q) with cheaper computational cost and 
still with good convergence properties. 
One possibility is to define hierarchies of relaxations of (Q)
based on other positivity certificates. Such alternative positivity certificates are in general weaker but their implementation is much easier as it results in LP-relaxations (e.g. as in DSOS \cite{ahmadi2014dsos}),  second-order cone relaxations (e.g., as in SDSOS \cite{ahmadi2014dsos}), or cheaper SDP-relaxations (e.g. as in BSOS \cite{TohLasserre}). 

Another possibility is to address \emph{sparsity} often present in the description of large-scale instances of (Q). One classical approach is to consider  so-called \emph{correlative sparsity} patterns,  developed in \cite{Las06,waki}. This is represented by the correlative sparsity pattern (csp) graph which grasps the links between variables. Concretely, the nodes of a csp graph correspond to the variables and there is an edge between two nodes (variables) if and only if these two variables appear in the same term of the objective polynomial $f$ or appear in the same polynomial $g_j$ involved in the description of $\mathbf{K}$. One then partitions the variables into blocks according to the maximal cliques of the chordal extension of the csp graph to obtain a moment-SOS hierarchy for (\textrm{Q}) with quasi block-diagonal SDP matrices. If each maximal clique has a small size, this will significantly reduce the computational cost. This approach has been successively applied for solving optimal powerflow problems~\cite{Josz16}, roundoff error bound analysis~\cite{toms17,toms18}, or more recently to approximate the volume of sparse semialgebraic sets~\cite{Tacchi19}.

Nevertheless many POPs that are fairly sparse do not fulfill the correlative sparsity pattern. For instance, if $f$ has a term involving all variables or some constraint $g_j$ (e.g. $\Vert\x\Vert^2$) involves all variables, then the correlative sparsity pattern fails. Besides, even if a POP admits a correlative sparsity pattern, some maximal clique of the csp graph may still have a big size (like over $20$ variables), which makes the resulting SDP problem  still hard to solve.

However, instead of exploiting sparsity from the perspective of \emph{variables}, one can also exploit sparsity from the perspective of 
\emph{terms} as described in \cite{wang,wang2}. This is the route followed in this paper.
\subsection*{Novelty with respect to \cite{wang2}} In \cite{wang2}, we exploit the term-sparsity occurring in the description of (Q) so as to define a sparsity-adapted version of the moment-SOS hierarchy, which scales much better with the size of the initial problem (Q). 
Roughly speaking, the sparsity considered in \cite{wang2} can be also represented by a graph, which is called a {\em term-sparsity pattern (tsp) graph}. 
But unlike the csp graph, the nodes of a tsp graph correspond to monomials (not variables) and the edges of the graph grasp the links between monomials in the SOS representation of positive polynomials. 
In \cite{wang2}, we design an iterative procedure to enlarge the tsp graph in order to iteratively exploit the term-sparsity in (\textrm{Q}). 
Each iteration consists of two steps: (i) a support-extension operation and (ii) a block-closure operation on adjacency matrices.

\newtssos{We first propose to replace the second step from \cite{wang2} by a chordal-extension operation.}
In doing so we obtain a sequence
\begin{equation*}
    G_1\subseteq G_2\subseteq \cdots\subseteq G_r,
\end{equation*}
of graphs where ``$G_{i}\subseteq G_{i+1}$'' means that $G_{i}$ is a subgraph of $G_{i+1}$.
\newtssos{The main difference with \cite{wang2} is that (ii) now consists of performing an (approximately) minimal chordal extension instead of performing completion of the connected components of each graph.}
Then combining this iterative procedure with the standard moment-SOS hierarchy results in a two-level moment-SOS hierarchy with quasi block-diagonal SDP matrices. When the size of blocks is small, then the associated SDP relaxations are drastically much cheaper to solve.

To some extent, the term-sparsity (focusing on monomials) is finer than the correlative sparsity (focusing on variables). If a POP is sparse in the sense of correlative sparsity, which means that the csp graph is not complete, then it must be sparse in the sense of term-sparsity, which means that the tsp graph is not complete, while the converse is not necessarily true. So the basic idea for solving large-scale POPs is as follows: first apply correlative sparsity to obtain a coarse decomposition in terms of variables with cliques, and second apply term-sparsity to the cliques that still have a big size.

\subsection*{Contribution}
We provide a new sparse moment-SOS framework based on \emph{term-sparsity} and \emph{chordal graphs}, following the route of our previous paper ~\cite{wang2}. 
It is in deep contrast to the approach based on the sole correlative sparsity and provides a new item in the arsenal of sparsity-exploiting techniques for moment-SOS hierarchies of POPs. More precisely:

$\bullet$ We provide an iterative procedure that exploits term-sparsity in POPs.
The case of unconstrained polynomial optimization is treated in \S~\ref{sec3} and the case of constrained polynomial optimization is treated in \S~\ref{sec4}.
It results in a two-level moment-SOS hierarchy that we  \newtssos{call the {\em chordal-TSSOS hierarchy} (as the ``TSSOS'' terminology was used in our prior work \cite{wang2}). The resulting SDP has {\em quasi block-diagonal} SDP matrices, which is the crucial feature of the 
chordal-TSSOS hierarchy.}


$\bullet$ In \S~\ref{sec6} we provide a computational cost  estimate for the first step   of the chordal-TSSOS hierarchy 
via a careful analysis of the structure of tsp graphs.

$\bullet$ In \S~\ref{sec:benchs} we provide various numerical experiments to illustrate that POPs with significantly large size (up to $200$ variables) and without correlative sparsity,  can be solved by our chordal-TSSOS hierarchy. 

We emphasize that in \emph{all} numerical examples tested in this paper (except the Broyden banded function from \cite{waki}), the usual correlative sparsity pattern is dense or almost dense and so yields no or little computational savings (or cannot even be implemented). 
Concerning Broyden banded functions, even though both tsp and csp apply,  
tsp still yields blocks whose size is smaller than those obtained with csp, and thus results in lower computational cost for the corresponding SDPs. 

\newtssos{Therefore the chordal-TSSOS hierarchy should be considered as 
a \emph{complement} to TSSOS \cite{wang2} rather than just a \emph{variant}. Indeed
on the one hand, TSSOS has a \emph{guaranteed convergence property} with good efficiency reported  in \cite{wang2} when compared to other hierarchies. 
On the other hand, while chordal-TSSOS lacks theoretical convergence guarantee, in practice it has 
superior performance (and with finite convergence in many cases).  
Hence a user should start with chordal-TSSOS and possibly turns to TSSOS if the convergence does not take place.}

\section{Notation and Preliminaries}
\subsection{SOS polynomials}\label{intro1}
Let $\x=(x_1,\ldots,x_n)$ be a tuple of variables and $\R[\x]=\R[x_1,\ldots,x_n]$ be the ring of real $n$-variate polynomials. For a subset $\A\subseteq\N^n$, we denote by $\Conv(\A)$ the convex hull of $\A$. A polynomial $f\in\R[\x]$ can be written as $f(\x)=\sum_{\a\in\A}f_{\a}\x^{\a}$ with $\A\subseteq\N^n$ and $f_{\a}\in\R, \x^{\a}=x_1^{\alpha_1}\cdots x_n^{\alpha_n}$. The support of $f$ is defined by $\supp(f)=\{\a\in\A\mid f_{\a}\ne0\}$, and the Newton polytope of $f$ is defined as $\New(f)=\Conv(\{\a:\a\in\supp(f)\})$. We use $|\cdot|$ to denote the cardinality of a set.

For a finite set $\A\subseteq\N^n$, 
let $\x^{\mathscr{A}}$ be the $|\mathscr{A}|$-dimensional column vector consisting of elements $\x^{\a},\a\in\mathscr{A}$ (fix any ordering on $\N^n$). For a positive integer $r$, the set of $r\times r$ symmetric matrices is denoted by $\mathbf{S}^r$ and the set of $r\times r$ positive semidefinite (PSD) matrices is denoted by $\mathbf{S}_+^r$. Let us denote by $\langle A, B\rangle\in\R$ the trace inner-product, defined by $\langle A, B\rangle=\Tr(A^TB)$.

Given a polynomial $f(\x)\in\R[\x]$, if there exist polynomials $f_1(\x),\ldots,f_t(\x)$ such that
\begin{equation}\label{sec2-eq1}
f(\x)=\sum_{i=1}^tf_i(\x)^2,
\end{equation}
then we say that $f(\x)$ is a {\em sum of squares} (SOS) polynomial. Clearly, the existence of an SOS decomposition of a given polynomial provides a certificate for its global nonnegativity. For $d\in\N$, let $\N^n_d:=\{\a=(\alpha_i)\in\N^n\mid\sum_{i=1}^n\alpha_i\le d\}$. Assume that $f\in\R[\x]$ is a polynomial of degree $2d$. If we choose the standard monomial basis $\x^{\N^n_{d}}$, then the SOS condition \eqref{sec2-eq1} is equivalent to the existence of a PSD matrix $Q$, which is called a {\em Gram matrix} \cite{re2}, such that
\begin{equation}\label{sec2-eq2}
f(\x)=(\x^{\N^n_{d}})^TQ\x^{\N^n_{d}}.
\end{equation}
When $f$ is sparse, the size of the corresponding SDP problem (\ref{sec2-eq2}) can be reduced by computing a smaller monomial basis. In fact, the set $\N^n_{d}$ in (\ref{sec2-eq2}) can be replaced by the integer points in half of the Newton polytope of $f$, i.e., by
\begin{equation}\label{sec2-eq3}
\B=\frac{1}{2}\cdot\New(f)\cap\N^n\subseteq\N^n_{d}.
\end{equation}
See \cite{re} for a proof.
We refer to this as the \emph{Newton polytope method}.
For convenience, we abuse notation in the sequel and denote by $\B$ instead of $\x^{\B}$ a monomial basis.

For a polynomial $f(\x)=\sum_{\a\in\A}f_{\a}\x^{\a}$ with $\supp(f)=\A$, let $\mathscr{B}$ be a monomial basis. For any $\a\in\mathscr{B}+\mathscr{B}:=\{\b+\g\mid\b,\g\in\B\}$, associate it with a $(0,1)$-binary matrix $A_{\a}\in\mathbf{S}^{|\B|}$ such that $[A_{\a}]_{\b\g}=1$ iff $\b+\g=\a$ for all $\b,\g\in\B$. Then $f$ is an SOS polynomial iff there exists $Q\in\mathbf{S}_+^{|\B|}$ such that the following coefficient matching condition holds:
\begin{equation}\label{sec2-eq4}
\langle A_{\a}, Q\rangle=f_{\a} \textrm{ for all }\a\in\mathscr{B}+\mathscr{B},
\end{equation}
where we set $f_{\a}=0$ if $\a\notin\A$.

\subsection{Moment-SOS relaxations for POPs}
With $\y=(y_{\a})$ being a sequence indexed by the standard monomial basis $\{\x^{\a}\}$ of $\R[\x]$, let $L_{\y}:\R[\x]\rightarrow\R$ be the linear functional
\begin{equation*}
f=\sum_{\a}f_{\a}\x^{\a}\mapsto L_{\y}(f)=\sum_{\a}f_{\a}y_{\a}.
\end{equation*}
Given a monomial basis $\B$, the {\em moment} matrix $M_{\B}(\y)$  associated with $\B$ and $\y$ is the matrix with rows and columns indexed by $\B$ such that
\begin{equation*}
M_{\B}(\y)_{\b\g}:=L_{\y}(\x^{\b}\x^{\g})=y_{\b+\g}, \quad\forall\b,\g\in\B.
\end{equation*}
If $\B$ is the standard monomial basis $\N^n_{d}$, we also denote $M_{\B}(\y)$ by $M_{d}(\y)$.

Consider the unconstrained polynomial optimization problem:
\begin{equation}\label{upop}
(\textrm{P}_0):\quad\lambda^*:=\inf_{\x}\{f(\x):\x\in\R^n\}
\end{equation}
with $f(\x)\in\R[\x]$ of degree $2d$. Let $\B$ be a monomial basis. The moment SDP relaxation of $(\textrm{P}_0)$ is (\cite{las1})
\begin{equation}\label{sec4-eq3}
(\textrm{P}):\quad
\begin{array}{rl}
\lambda_{sos}:=\inf&\{\,L_{\y}(f)\,:\\
\textrm{s.t.}\quad &M_{\B}(\y)\succeq0,\\
&y_{\mathbf{0}}=1\,\}\,.
\end{array}
\end{equation}
The dual SDP problem of~\eqref{sec4-eq3} is
\begin{equation}\label{sec2-usos}
(\textrm{P})^*:\quad
\begin{array}{rl}
\sup&\{\,\lambda\,:\\
\textrm{s.t.}\quad &\langle Q,A_{\a}\rangle+\lambda\delta_{\mathbf{0}\a}=f_{\a},\quad\forall\a\in\N^n_{2d},\\
&Q\succeq0\,\}\,.
\end{array}
\end{equation}

Suppose $g=\sum_{\a}g_{\a}\x^{\a}\in\R[\x]$ and let $\y=(y_{\a})$ be given. For a positive integer $d$, the {\em localizing} matrix $M_{d}(g\y)$ associated with $g$ and $\y$ is the matrix with rows and columns indexed by $\N^n_{d}$ such that
\begin{equation*}
M_{d}(g\y)_{\b\g}:=L_{\y}(g\x^{\b}\x^{\g})=\sum_{\a}g_{\a}y_{\a+\b+\g}, \quad\forall\b,\g\in\N^n_{d}.
\end{equation*}

Consider the constrained polynomial optimization problem:
\begin{equation}\label{sec2-eq9}
(\textrm{Q}_0):\quad\lambda^*:=\inf_{\x}\{f(\x) : \x\in\mathbf{K}\} \,,
\end{equation}
where $f(\x)\in\R[\x]$ is a polynomial and $\mathbf{K}\subseteq\R^{n}$ is the basic semialgebraic set
\begin{equation}\label{sec2-eq10}
\mathbf{K} = \{\x\in\R^{n} : g_j(\x)\ge 0, j = 1,\ldots,m\},
\end{equation}
for some polynomials $g_j(\x)\in\R[\x], j = 1,\ldots,m.$

Let $d_j:=\lceil\deg(g_j)/2\rceil,j=1,\ldots,m$ and let $\hat{d}\ge\max\{\lceil\deg(f)/2\rceil,d_1,\ldots,d_m\}$ be a positive integer. Then the Lasserre's hierarchy indexed by $\hat{d}$ of primal moment SDP relaxations of ($\textrm{Q}_0$) is defined by (\cite{las1}):
\begin{equation}\label{sec2-eq11}
(\textrm{Q}_{\hat{d}}):\quad
\begin{array}{rl}
\lambda_{\hat{d}}:=\inf &\{\,L_{\y}(f)\,:\\
\textrm{s.t.}&M_{\hat{d}}(\y)\succeq0,\\
&M_{\hat{d}-d_j}(g_j\y)\succeq0,\quad j=1,\ldots,m,\\
&y_{\mathbf{0}}=1\,\}\,.
\end{array}
\end{equation}
We call $\hat{d}$ the {\em relaxation order}. 

Set $g_0:=1$ and $d_0:=0$. For each $j$, writing $M_{\hat{d}-d_j}(g_j\y)=\sum_{\a}D_{\a}^jy_{\a}$ for appropriate $(0,1)$-binary matrices $\{D_{\a}^j\}$, we can write the dual of \eqref{sec2-eq11} as
\begin{equation}\label{sec2-eq12}
(\textrm{Q}_{\hat{d}})^*:\quad
\begin{array}{rl}
\sup&\{\,\lambda\,:\\
\textrm{s.t.}&\displaystyle\sum_{j=0}^m\langle Q_j,D_{\a}^j\rangle+\lambda\delta_{\mathbf{0}\a}=f_{\a},\quad\forall\a\in\N^n_{2\hat{d}},\\
&Q_j\succeq0,\quad j=0,\ldots,m\,\}\,,
\end{array}
\end{equation}
where $\delta_{\mathbf{0}\a}$ is the usual Kronecker symbol.

\subsection{Chordal graphs and sparse matrices}
We introduce some basic notions from graph theory. An {\em (undirected) graph} $G(V,E)$ or simply $G$ consists of a set of nodes $V$ and a set of edges $E\subseteq\{\{v_i,v_j\}\mid (v_i,v_j)\in V\times V\}$. Note that we admit self-loops (i.e.\,edges that connect the same node) in the edge set $E$. If $G$ is a graph, we also use $V(G)$ and $E(G)$ to indicate the set of nodes of $G$ and the set of edges of $G$, respectively. For two graphs $G,H$, we say that $G$ is a {\em subgraph} of $H$ if $V(G)\subseteq V(H)$ and $E(G)\subseteq E(H)$, denoted by $G\subseteq H$. For a graph $G(V,E)$, a {\em cycle} of length $k$ is a set of nodes $\{v_1,v_2,\ldots,v_k\}\subseteq V$ with $\{v_k,v_1\}\in E$ and $\{v_i, v_{i+1}\}\in E$, for $i=1,\ldots,k-1$. A {\em chord} in a cycle $\{v_1,v_2,\ldots,v_k\}$ is an edge $\{v_i, v_j\}$ that joins two nonconsecutive nodes in the cycle.

A graph is called a {\em chordal graph} if all its cycles of length at least four have a chord. Chordal graphs include some common classes of graphs, such as complete graphs, line graphs and trees, and have applications in sparse matrix theory. Note that any non-chordal graph $G(V,E)$ can always be extended to a chordal graph $\overline{G}(V,\overline{E})$ by adding appropriate edges to $E$, which is called a {\em chordal extension} of $G(V,E)$. A {\em clique} $C\subseteq V$ of $G$ is a subset of nodes where $\{v_i,v_j\}\in E$ for any $v_i,v_j\in C$. If a clique $C$ is not a subset of any other clique, then it is called a {\em maximal clique}. It is known that maximal cliques of a chordal graph can be enumerated efficiently in linear time in the number of nodes and edges of the graph. See for example \cite{bp,fg,go} for the details.

Given a graph $G(V,E)$, a symmetric matrix $Q$ with row and column indices labeled by $V$ is said to have sparsity pattern $G$ if $Q_{\b\g}=Q_{\g\b}=0$ whenever $\{\b,\g\}\notin E$. Let $\mathbf{S}_G$ be the set of symmetric matrices with sparsity pattern $G$. The PSD matrices with sparsity pattern $G$ form a convex cone
\begin{equation}\label{sec2-eq5}
\mathbf{S}_+^{|V|}\cap\mathbf{S}_G=\{Q\in\mathbf{S}_G\mid Q\succeq0\}.
\end{equation}

Given a maximal clique $C$ of $G(V,E)$, we define a matrix $P_{C}\in \R^{|C|\times |V|}$ as
\begin{equation}\label{sec2-eq6}
(P_{C})_{i\b}=\begin{cases}
1, &\textrm{if }C(i)=\b,\\
0, &\textrm{otherwise}.
\end{cases}
\end{equation}
where $C(i)$ denotes the $i$-th node in $C$, sorted in the ordering compatibly with $V$. Note that $Q_{C}=P_{C}QP_{C}^T\in \mathbf{S}^{|C|}$ extracts a principal submatrix $Q_C$ defined by the indices in the clique $C$ from a symmetry matrix $Q$, and $Q=P_{C}^TQ_{C}P_{C}$ inflates a $|C|\times|C|$ matrix $Q_{C}$ into a sparse $|V|\times |V|$ matrix $Q$.

When the sparsity pattern graph $G$ is chordal, the cone $\mathbf{S}_+^{|V|}\cap\mathbf{S}_G$ can be
decomposed as a sum of simple convex cones, as stated in the following theorem.
\begin{theorem}[\cite{va}, Theorem 9.2]\label{sec2-thm}
Let $G(V,E)$ be a chordal graph and assume that $C_1,\ldots,C_t$ are all of the maximal cliques of $G(V,E)$. Then a matrix $Q\in\mathbf{S}_+^{|V|}\cap\mathbf{S}_G$ if and only if there exist $Q_{k}\in \mathbf{S}_+^{|C_k|}$ for $k=1,\ldots,t$ such that $Q=\sum_{k=1}^tP_{C_k}^TQ_{k}P_{C_k}$.
\end{theorem}

Given a graph $G(V,E)$, let $\Pi_{G}$ be the projection from $\mathbf{S}^{|V|}$ to the subspace $\mathbf{S}_G$, i.e., for $Q\in\mathbf{S}^{|V|}$,
\begin{equation}\label{sec2-eq7}
\Pi_{G}(Q)_{\b\g}=\begin{cases}
Q_{\b\g}, &\textrm{if }\{\b,\g\}\in E,\\
0, &\textrm{otherwise}.
\end{cases}
\end{equation}

We denote by $\Pi_{G}(\mathbf{S}_+^{|V|})$ the set of matrices in $\mathbf{S}_G$ that have a PSD completion, i.e., 
\begin{equation}\label{sec2-eq8}
\Pi_{G}(\mathbf{S}_+^{|V|})=\{\Pi_{G}(Q)\mid Q\in\mathbf{S}_+^{|V|}\}.
\end{equation}

One can check that the PSD completable cone $\Pi_{G}(\mathbf{S}_+^{|V|})$ and the PSD cone $\mathbf{S}_+^{|V|}\cap\mathbf{S}_G$ form a pair of dual cones in $\mathbf{S}_G$. Moreover, for a chordal graph $G$, the decomposition result for the cone $\mathbf{S}_+^{|V|}\cap\mathbf{S}_G$ in Theorem \ref{sec2-thm} leads to the following characterization of the PSD completable cone $\Pi_{G}(\mathbf{S}_+^{|V|})$.
\begin{theorem}[\cite{va}, Theorem 10.1]\label{sec2-thm2}
Let $G(V,E)$ be a chordal graph and assume that $C_1,\ldots,C_t$ are all of the maximal cliques of $G(V,E)$. Then a matrix $Q\in\Pi_{G}(\mathbf{S}_+^{|V|})$ if and only if $Q_{k}=P_{C_k}QP_{C_k}^T\succeq0$ for $k=1,\ldots,t$.
\end{theorem}

For more details about sparse matrices and chordal graphs, the reader may refer to \cite{va}.

\section{\newtssos{The chordal-TSSOS Hierarchy: Unconstrained case}}\label{sec3}
In this section, we describe an iterative procedure to exploit term-sparsity for the primal \eqref{sec4-eq3} and dual \eqref{sec2-usos} SDP relaxations of unconstrained POPs.

For a polynomial $f(\x)=\sum_{\a\in\A}f_{\a}\x^{\a}\in\R[\x]$ with $\supp(f)=\A$ (assuming $\mathbf{0}\in\A$), let $\mathscr{B}$ be a monomial basis.
In the following, we will consider graphs with $V:=\mathscr{B}$ as set of nodes. Suppose $G(V,E)$ is such a graph. We define the {\em support} of $G$ by
$$\supp(G):=\{\b+\g\mid\{\b,\g\}\in E\}.$$
We further define two operations on $G$: {\em support-extension} and {\em chordal-extension}. The support-extension of $G$, denoted by $\SE(G)$, is the graph with set of nodes $\mathscr{B}$ and with the edge set
$$E(\SE(G)):=\{\{\b,\g\}\mid\b+\g\in\supp(G)\}.$$
\begin{example}
Consider the following graph $G(V,E)$ with $$V=\{1,x_1,x_2,x_3,x_2x_3,x_1x_3,x_1x_2\} \textrm{ and } E=\{\{1,x_2x_3\},\{x_2,x_1x_3\}\}.$$
Then $E(\SE(G))=\{\{1,x_2x_3\},\{x_2,x_1x_3\},\{x_2,x_3\},\{x_1,x_2x_3\},\{x_3,x_1x_2\}\}$. See Figure \ref{support} for the support-extension $\SE(G)$ of $G$.
\begin{figure}[htbp]
\caption{The support-extension $\SE(G)$ of $G$}\label{support}
\begin{center}
{\tiny
\begin{tikzpicture}[every node/.style={circle, draw=blue!50, thick, minimum size=8mm}]
\node (n1) at (0,0) {$1$};
\node (n2) at (2,0) {$x_1$};
\node (n3) at (4,0) {$x_2$};
\node (n4) at (6,0) {$x_3$};
\node (n5) at (1,-2) {$x_2x_3$};
\node (n6) at (4,-2) {$x_1x_3$};
\node (n7) at (6,-2) {$x_1x_2$};
\draw (n1)--(n5);
\draw (n3)--(n6);
\draw[dashed] (n2)--(n5);
\draw[dashed] (n3)--(n4);
\draw[dashed] (n4)--(n7);
\end{tikzpicture}}\\
{\small The dashed edges are added in the process of the support-extension.}
\end{center}
\end{figure}
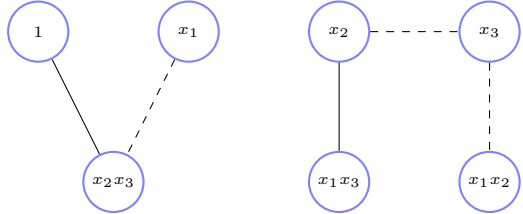 
\end{example}

\newtssos{Any specific chordal-extension of $G$ is denoted by $\overline{G}$.}
\begin{example}
Consider the following graph $G(V,E)$ with $V=\{x_1,x_2,x_3,x_4,x_5,x_6\}$ and $E=\{\{x_1,x_2\},\{x_2,x_3\},\{x_3,x_4\},\{x_4,x_5\},\{x_5,x_6\},\{x_6,x_1\}\}.$
See Figure \ref{chordal} for the chordal-extension $\overline{G}$ of $G$.
\begin{figure}[htbp]
\caption{The chordal-extension $\overline{G}$ of $G$}\label{chordal}
\begin{center}
{\tiny
\begin{tikzpicture}[every node/.style={circle, draw=blue!50, thick, minimum size=7.5mm}]
\node (n2) at (90:2) {$x_1$};
\node (n3) at (30:2) {$x_2$};
\node (n4) at (330:2) {$x_3$};
\node (n5) at (270:2) {$x_4$};
\node (n6) at (210:2) {$x_5$};
\node (n1) at (150:2) {$x_6$};
\draw (n2)--(n3);
\draw[dashed] (n2)--(n4);
\draw[dashed] (n2)--(n5);
\draw[dashed] (n2)--(n6);
\draw (n3)--(n4);
\draw (n4)--(n5);
\draw (n5)--(n6);
\draw (n6)--(n1);
\draw (n1)--(n2);
\end{tikzpicture}}\\
{\small The dashed edges are added in the process of the chordal-extension.}
\end{center}
\end{figure}
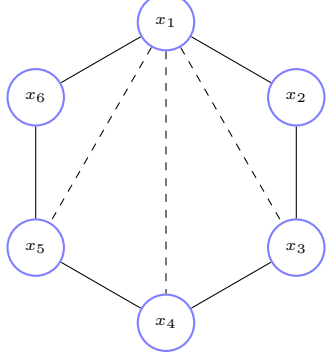 
\end{example}

\begin{remark}
For a graph $G(V,E)$, the chordal extension of $G$ is usually not unique. A chordal extension with the least number of edges is called a {\em minimum chordal extension}. Finding a minimum chordal extension of a graph is an NP-hard problem in general. Fortunately, several heuristic algorithms, such as the minimum degree ordering, are known to efficiently produce a good approximation \cite{am,be,heg}. In this paper, we actually use an approximately minimum chordal extension for a graph.
\end{remark}

In the following, we assume that for graphs $G,H$ with same set of nodes, if $E(G)\subseteq E(H)$, then $E(\overline{G})\subseteq E(\overline{H})$. This assumption is reasonable since any chordal extension of $H$ must be also a chordal extension of $G$.

Let $f(\x)\in\R[\x]$ with $\supp(f)=\A$ and $\B$ a monomial basis with $r=|\B|$. We define $G_0(V,E_0)$ to be the graph with $V=\B$ and
$$E_0=\{\{\b,\g\}\mid\b+\g\in\A\cup(2\B)\},$$
where $2\B=\{2\b\mid\b\in\B\}$. We call $G_0$ the {\em term-sparsity pattern (tsp) graph} associated with $f$.

For $k\ge1$, we recursively define a sequence of graphs $\{G_k(V,E_k)\}_{k\ge1}$ by
\begin{equation}\label{sec3-graph}
G_k:=\overline{\SE(G_{k-1})}.
\end{equation}

If $f$ is sparse, by replacing $M_{\B}(\y)\succeq0$ with the weaker condition $M_{\B}(\y)\in\Pi_{G_k}(\mathbf{S}_+^{r})$ in \eqref{sec4-eq3}, we obtain a sparse moment SDP relaxation of $(\textrm{P}_0)$ \eqref{upop} for each $k\ge1$:
\begin{equation}\label{sec4-ukmom}
(\textrm{P}^k):\quad
\begin{array}{rl}
\lambda_k:=\inf &\{\,L_{\y}(f)\,:\\
\textrm{s.t.}&M_{\B}(\y)\in\Pi_{G_k}(\mathbf{S}^{r}_+),\\
&y_{\mathbf{0}}=1\,\}\,.
\end{array}
\end{equation}
We call $k$ the {\em sparse order}. By construction, we have $G_{k}\subseteq G_{k+1}$ for all $k\ge1$ and therefore the sequence of graphs $\{G_k(V,E_k)\}_{k\ge1}$ stabilizes after a finite number of steps. 

\begin{theorem}
The sequence $\{\lambda_k\}_{k\ge 1}$ is monotone nondecreasing and $\lambda_k\le\lambda_{sos}$ for all $k$.
\end{theorem}
\begin{proof}
By construction, each maximal clique of $G_k$ is a subset of some maximal clique of $G_{k+1}$. Thus by Theorem \ref{sec2-thm2}, we have that $(\textrm{P}^k)$ is a relaxation of $(\textrm{P}^{k+1})$ (and also a relaxation of (P)). This yields the desired conclusions.
\end{proof}

As a consequence, we have the following hierarchy of lower bounds for the optimum of the original problem ($\textrm{P}_0$):
\begin{equation}\label{cliquehier}
\lambda^*\ge\lambda_{sos}\ge\cdots\ge\lambda_2\ge\lambda_1.
\end{equation}
\newtssos{We say that \eqref{sec4-ukmom} (and its associated sequence \eqref{cliquehier})} is the {\em \newtssos{chordal}-TSSOS} hierarchy for ($\textrm{P}_0$).

Unlike for the block moment-SOS hierarchy developed in \cite{wang2} (\newtssos{which will be referred to as the block-TSSOS hierarchy in this paper}), \newtssos{in theory there is no guarantee that the hierarchy of lower bounds  $\{\lambda_k\}_{k\ge1}$ converges to the value $\lambda_{sos}$}. The following is an example.
\begin{example}\label{ex1}
Consider the polynomial $f=x_1^2-2x_1x_2+3x_2^2-2x_1^2x_2+2x_1^2x_2^2-2x_2x_3+6x_3^2+18x_2^2x_3-54x_2x_3^2+142x_2^2x_3^2$ (\cite{ncsparse}). A monomial basis computed by the Newton polytope method is $\{1,x_1,x_2,x_3,x_1x_2,x_2x_3\}$.
Figure \ref{tsp} shows the tsp graph $G_0$ (without dashed edges) and its chordal-extension $G_1$ (with dashed edges) for $f$. The graph sequence $\{G_k\}_{k\ge1}$ stabilizes at $k=1$. Solving the SDP problem ($\textrm{P}^1$) associated with $G_1$, we obtain $\lambda_1=-0.00355$ while we have $\lambda_{sos}=\lambda^*=0$. 

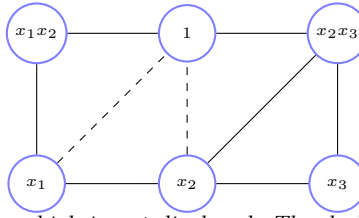
\begin{figure}[htbp]
\caption{The tsp graph $G_0$ and its chordal-extension $G_1$ for Example \ref{ex1}}\label{tsp}
{\tiny
\begin{center}
\begin{tikzpicture}[every node/.style={circle, draw=blue!50, thick, minimum size=7.5mm}]
\node (n2) at (0,0) {$x_1$};
\node (n3) at (2,0) {$x_2$};
\node (n4) at (4,0) {$x_3$};
\node (n5) at (0,2) {$x_1x_2$};
\node (n1) at (2,2) {$1$};
\node (n6) at (4,2) {$x_2x_3$};
\draw (n2)--(n3);
\draw (n3)--(n4);
\draw (n6)--(n4);
\draw (n1)--(n6);
\draw (n1)--(n5);
\draw (n2)--(n5);
\draw (n3)--(n6);
\draw[dashed] (n2)--(n1);
\draw[dashed] (n3)--(n1);
\end{tikzpicture}\\
{\small Each node has a self-loop which is not displayed. The dashed edges are added in the process of the chordal-extension.}
\end{center}}
\end{figure}
\end{example}


\begin{remark}
Even though \newtssos{there is no theoretical guarantee that} the sequence of optimal solutions of the \newtssos{chordal-TSSOS} hierarchy converges to the optimal solution of the corresponding dense moment-SOS relaxation for ($\textrm{P}_0$), in practice \newtssos{finite convergence takes place} in many cases as we shall see in \newtssos{numerical experiments}. Besides, by mixing the \newtssos{chordal-TSSOS} hierarchy with the block-TSSOS hierarchy in \cite{wang2}, we can easily ensure convergence.
\end{remark}

For each $k\ge1$, the dual SDP problem of~\eqref{sec4-ukmom} is
\begin{equation}\label{sec4-uksos}
(\textrm{P}^k)^*:\quad
\begin{cases}
\sup\quad &\lambda\\
\textrm{s.t.}\quad &\langle Q,A_{\a}\rangle+\lambda\delta_{\mathbf{0}\a}=f_{\a},\quad\forall\a\in\supp(G_k),\\
&Q\in\mathbf{S}_+^{r}\cap\mathbf{S}_{G_k},
\end{cases}
\end{equation}
where $A_{\a}$ is defined by \eqref{sec2-eq4}.

\begin{proposition}
For each $k\ge1$, there is no duality gap between ($\textrm{P}^k$) and $(\textrm{P}^k)^*$.
\end{proposition}
\begin{proof}
This easily follows from Proposition 3.1 of \cite{las1} for the dense case and Theorem \ref{sec2-thm2}.
\end{proof}

\subsection*{\newtssos{Comparison with SDSOS \cite{ahmadi2014dsos}}} The following definition of SDSOS polynomials has been introduced and studied in \cite{ahmadi2014dsos}. 
A symmetric matrix $Q\in\mathbf{S}^r$ is {\em diagonally dominant} if $Q_{ii}\ge\sum_{j=1}^r|Q_{ij}|$ for $i=1,\ldots,r$ and is {\em scaled diagonally dominant} if there exists a positive definite $r\times r$ diagonal matrix $D$ such that $DQD$ is diagonally dominant. 
We say that a polynomial $f(\x)\in\R[\x]$ is a {\em scaled diagonally dominant sum of squares} (SDSOS) polynomial if it admits a Gram matrix representation \eqref{sec2-eq2} with a scaled diagonally dominant Gram matrix $Q$. We denote the set of SDSOS polynomials by $SDSOS$.

\newtssos{Following \cite{ahmadi2014dsos}, by replacing the nonnegativity condition in ($\textrm{P}_0$) with the SDSOS condition, one obtains the SDSOS relaxation of (P) and ($\textrm{P}_0$):}
\begin{equation*}
(\textrm{SDSOS}):\quad\lambda_{sdsos}:=\sup_{\lambda}\{\lambda : f(\x)-\lambda\in\SDSOS\}.
\end{equation*}
\begin{theorem}\label{sec4-thm1}
\newtssos{With the above notation, $\lambda_1\ge\lambda_{sdsos}$.}
\end{theorem}
\begin{proof}
Let $f\in\R[\x]$ with $\A=\supp(f)$ and $\B$ a monomial basis with $r=|\B|$. Assume that $f\in\SDSOS$, i.e.,  $f$ admits a scaled diagonally dominant Gram matrix $Q\in\mathbb{S}_+^r$ indexed by $\B$. We then construct a Gram matrix $\tilde{Q}$ for $f$ by
\begin{equation*}
\tilde{Q}_{\b\g}=
\begin{cases}
0,\quad&\textrm{if } \b+\g\notin\A\cup2\B, \\
Q_{\b\g},\quad&\textrm{otherwise}.
\end{cases}
\end{equation*}
It is easy to see that we still have $f=(\x^{\mathscr{B}})^T\tilde{Q}\x^{\mathscr{B}}$. Note that we only replace off-diagonal entries by zeros in $Q$ and replacing off-diagonal entries by zeros does not affect the scaled diagonal dominance of a matrix. Hence $\tilde{Q}$ is also a scaled diagonally dominant matrix. Moreover, we have $\tilde{Q}\in\mathbf{S}_+^{r}\cap\mathbf{S}_{G_1}$ by construction. It follows that $(\textrm{SDSOS})$ is a relaxation of $(\textrm{P}^1)^*$. Hence $\lambda_1\ge\lambda_{sdsos}$.
\end{proof}

The next result states that $\lambda_1=\lambda_{sos}$ always holds in the quadratic case.
\begin{theorem}\label{sec3-thm2}
Suppose that the objective function $f\in\R[\x]$ in ($\textrm{P}_0$) is a quadratic polynomial. Then $\lambda_1=\lambda_{sos}$.
\end{theorem}
\begin{proof}
Assume $\supp(f)=\A$. Since $f$ is quadratic, we take $\B=\{\be_0,\be_1,\ldots,\be_n\}$ as a monomial basis, where $\be_0=\mathbf{0}$ and $\{\be_i\}_{i=1}^n$ is the standard basis of $\R^n$. Let $G_0$ be the tsp graph associated with $f$.
We only need to prove that if $f$ admits a PSD Gram matrix, then $f$ admits a Gram matrix in $\mathbf{S}_+^{n+1}\cap\mathbf{S}_{G_0}$.
Suppose that $Q=[q_{ij}]_{i,j=0}^n$ is a PSD Gram matrix of $f$ indexed by $\B$. Note that for all $i,j$, if $\{\be_i,\be_j\}\not\in E(G_0)$, then we have $\be_i+\be_j\notin\A$, which implies $q_{ij}=0$. It follows that $Q\in\mathbf{S}_{G_0}$ as desired.
\end{proof}

\section{\newtssos{The chordal-TSSOS Hierarchy: Constrained case}}\label{sec4}
In this section, we describe an iterative procedure to exploit term-sparsity for the primal-dual moment-SOS hierarchy \eqref{sec2-eq11}-\eqref{sec2-eq12} of the constrained POP $(\textrm{Q}_0)$
\newtssos{defined in \eqref{sec2-eq9}-\eqref{sec2-eq10}}. Let
\begin{equation*}
\mathscr{A} = \supp(f)\cup\bigcup_{j=1}^m\supp(g_j).
\end{equation*}

Let $d_j:=\lceil\deg(g_j)/2\rceil,j=1,\ldots,m$ and $d:=\max\{\lceil\deg(f)/2\rceil,d_1,\ldots,d_m\}$. Fix a relaxation order $\hat{d}\ge d$ \newtssos{in} Lasserre's hierarchy \eqref{sec2-eq11}. Let $g_0=1$, $d_0=0$ and $\B_{j,\hat{d}}=\N^n_{\hat{d}-d_j}$ be the standard monomial basis for $j=0,\ldots,m$. We define a graph $G_{0,\hat{d}}^{(0)}(V_{0,\hat{d}},E_{0,\hat{d}}^{(0)})$ with $V_{0,\hat{d}}^{(0)}=\B_{0,\hat{d}}$ and
\begin{equation}\label{sec4-eq1}
E_{0,\hat{d}}^{(0)}=\{\{\b,\g\}\mid\b+\g\in\A\cup(2\B_{0,\hat{d}})\}.
\end{equation}
We call $G_0$ the {\em term-sparsity pattern (tsp) graph} associated with $(\textrm{Q}_0)$. 

For $k\ge1$, we recursively define a sequence of graphs $\{G_{j,\hat{d}}^{(k)}(V_{j,\hat{d}},E_{j,\hat{d}}^{(k)})\}_{k\ge1}$ with $V_{j,\hat{d}}=\B_{j,\hat{d}}$ for $j=0,\ldots,m$ by
\begin{equation}\label{sec4-graph}
G_{0,\hat{d}}^{(k)}:=\overline{\SE(G_{0,\hat{d}}^{(k-1)})}\textrm{ and }G_{j,\hat{d}}^{(k)}:=\overline{F_{j,\hat{d}}^{(k)}}, j=1,\ldots,m,
\end{equation}
where $F_{j,\hat{d}}^{(k)}$ is the graph with $V(F_{j,\hat{d}}^{(k)})=\B_{j,\hat{d}}$ and
\begin{equation}\label{sec4-eq2}
E(F_{j,\hat{d}}^{(k)})=\{\{\b,\g\}\mid(\supp(g_j)+\b+\g)\cap\supp(G_{0,\hat{d}}^{(k-1)})\ne\emptyset\}, j=1,\ldots,m.
\end{equation}

Let $r_j:=\binom{n+\hat{d}-d_j}{\hat{d}-d_j}$.
Therefore by replacing $M_{\hat{d}-d_j}(g_j\y)\succeq0$ with the weaker condition $M_{\hat{d}-d_j}(g_j\y)\in\Pi_{G_{j,\hat{d}}^{(k)}}(\mathbf{S}_+^{r_j})$  for $j=0,\ldots,m$ in \eqref{sec2-eq11}, we obtain the following sparse SDP relaxation of ($\textrm{Q}_{\hat{d}}$) and $(\textrm{Q}_0)$ for each $k\ge1$:
\begin{equation}\label{sec-eq1}
(\textrm{Q}_{\hat{d}}^k):\quad
\begin{array}{rl}
\lambda^{(k)}_{\hat{d}}:=\inf &\,\{\,L_{\y}(f)\,:\\
\textrm{s.t.}&M_{\hat{d}}(\y)\in\Pi_{G_{0,\hat{d}}^{(k)}}(\mathbf{S}_+^{r_0}),\\
&M_{\hat{d}-d_j}(g_j\y)\in\Pi_{G_{j,\hat{d}}^{(k)}}(\mathbf{S}_+^{r_j}),\quad j=1,\ldots,m,\\
&y_{\mathbf{0}}=1\,\}\,.
\end{array}
\end{equation}
We call $k$ the {\em sparse order}. By construction we have $G_{j,\hat{d}}^{(k)}\subseteq G_{j,\hat{d}}^{(k+1)}$ for all $j,k$. Therefore, for every $j$, the sequence of graphs
$\{G_{j,\hat{d}}^{(k)}\}_{k\ge1}$ stabilizes after a finite number of steps.

\begin{theorem}\label{sec6-thm1}
For fixed $\hat{d}\ge d$, the sequence $\{\lambda^{(k)}_{\hat{d}}\}_{k\ge1}$ is monotone nondecreasing and $\lambda^{(k)}_{\hat{d}}\le\lambda_{\hat{d}}$ for all $k$ \newtssos{(with $\lambda_{\hat{d}}$ as in \eqref{sec2-eq11}).}
\end{theorem}
\begin{proof}
By construction, for all $j,k$, each maximal clique of $G_{j,\hat{d}}^{(k)}$ is a subset of some maximal clique of $G_{j,\hat{d}}^{(k+1)}$. Hence by Theorem \ref{sec2-thm2}, $(\textrm{Q}_{\hat{d}}^{k})$ is a relaxation of $(\textrm{Q}_{\hat{d}}^{k+1})$ (and also a relaxation of $(\textrm{Q}_{\hat{d}})$). Therefore, $\{\lambda^{(k)}_{\hat{d}}\}_{k\ge1}$ is monotone nondecreasing and $\lambda^{(k)}_{\hat{d}}\le\lambda_{\hat{d}}$ for all $k$.
\end{proof}

\begin{theorem}\label{sec6-thm2}
For fixed $k\ge1$, the sequence $\{\lambda^{(k)}_{\hat{d}}\}_{\hat{d}\ge d}$ is monotone nondecreasing.
\end{theorem}
\begin{proof}
The conclusion follows if we can show that $G_{j,\hat{d}}^{(k)}\subseteq G_{j,\hat{d}+1}^{(k)}$ for all $j,\hat{d}$ since by Theorem \ref{sec2-thm2} this implies that $(\textrm{Q}_{\hat{d}}^{k})$ is a relaxation of $(\textrm{Q}_{\hat{d}+1}^{k})$. Let us prove $G_{j,\hat{d}}^{(k)}\subseteq G_{j,\hat{d}+1}^{(k)}$ by induction on $k$. For $k=1$, from $\eqref{sec4-eq1}$, we have $E_{0,\hat{d}}^{(0)}\subseteq E_{0,\hat{d}+1}^{(0)}$, which implies that $G_{j,\hat{d}}^{(1)}\subseteq G_{j,\hat{d}+1}^{(1)}$ for $j=0,\ldots,m$. Now assume that $G_{j,\hat{d}}^{(k)}\subseteq G_{j,\hat{d}+1}^{(k)}$, $j=0,\ldots,m$ hold for a given $k\geq 1$. Then from $\eqref{sec4-graph}$ and $\eqref{sec4-eq2}$ and by the induction hypothesis, we have $G_{j,\hat{d}}^{(k+1)}\subseteq G_{j,\hat{d}+1}^{(k+1)}$ for $j=0,\ldots,m$, which completes the induction and the proof.
\end{proof}

Combining Theorem \ref{sec6-thm1} and Theorem \ref{sec6-thm2}, we have the following two-level hierarchy of lower bounds for the optimum of $(\textrm{Q}_0)$:
\begin{equation}\label{cliquehierc}
\begin{matrix}
\lambda^{(1)}_{d}&\le&\lambda^{(2)}_{d}&\le&\cdots&\le&\lambda_{d}\\
\vge&&\vge&&&&\vge\\
\lambda^{(1)}_{d+1}&\le&\lambda^{(2)}_{d+1}&\le&\cdots&\le&\lambda_{d+1}\\
\vge&&\vge&&&&\vge\\
\vdots&&\vdots&&\vdots&&\vdots\\
\vge&&\vge&&&&\vge\\
\lambda^{(1)}_{\hat{d}}&\le&\lambda^{(2)}_{\hat{d}}&\le&\cdots&\le&\lambda_{\hat{d}}\\
\vge&&\vge&&&&\vge\\
\vdots&&\vdots&&\vdots&&\vdots\\
\end{matrix}
\end{equation}
\newtssos{The array of lower bounds \eqref{cliquehierc} (and its associated SDP-relaxations \eqref{sec-eq1}) is what we call the {\em chordal}-TSSOS moment-SOS hierarchy 
(in short {\em chordal}-TSSOS hierarchy) associated with $(\textrm{Q}_0)$}.

For each $k\ge1$, the dual of $(\textrm{Q}_{\hat{d}}^k)$ is
\begin{equation}\label{sec6-eq1}
(\textrm{Q}_{\hat{d}}^k)^*:
\begin{cases}
\sup\,&\lambda\\
\textrm{s.t.}\, &\sum_{j=0}^m\langle Q_j,D_{\a}^j\rangle+\lambda\delta_{\mathbf{0}\a}=f_{\a},\forall\a\in\cup_{j=0}^m(\supp(g_j)+\supp(G_{j,\hat{d}}^{(k)})),\\
&Q_j\in\mathbf{S}_+^{r_j}\cap\mathbf{S}_{G_{j,\hat{d}}^{(k)}},\quad j=0,\ldots,m,
\end{cases}
\end{equation}
where $D_{\a}^j$ is defined in Sec.~\ref{intro1}.

\begin{proposition}
Let $f\in\R[\x]$ and $\mathbf{K}$ be as in \eqref{sec2-eq10}. Assume that $K$ has a nonempty interior. Then there is no duality gap between $(\textrm{Q}_{\hat{d}}^k)$ and $(\textrm{Q}_{\hat{d}}^k)^*$ for any $\hat{d}\ge d$ and $k\ge1$.
\end{proposition}
\begin{proof}
By the duality theory of convex programming, this easily follows from Theorem 4.2 of \cite{las1} for the dense case and Theorem \ref{sec2-thm2}.
\end{proof}


\begin{remark}
As in the unconstrained case, \newtssos{there is no theoretical guarantee that the sequence of optimal values $\{\lambda_{\hat{d}}^{(k)}\}_{k\ge1}$ of the chordal-TSSOS hierarchy converges to the optimal value $\lambda_{\hat{d}}$ of the corresponding dense moment-SOS relaxation for ($\textrm{Q}_0$).} 
\newtssos{However, as observed in our numerical experiments, finite convergence takes place in many cases. Moreover, theoretical convergence is guaranteed if one mixes  the current ``clique-based" chordal-TSSOS hierarchy with the ``block-based'' TSSOS hierarchy from  \cite{wang2}.}
\end{remark}

\begin{remark}
As in Theorem \ref{sec3-thm2}, for quadratically constrained quadratic problems, we always have $\lambda^{(1)}_{1}=\lambda_{1}$.
\end{remark}

\begin{remark}
\newtssos{Our above treatment easily extends to include equality constraints.}
\end{remark}

\bigskip

\noindent{\bf Obtaining a possibly smaller initial monomial basis $\mathscr{B}$.}
\newtssos{The chordal-TSSOS hierarchy depends on the chosen initial monomial basis $\mathscr{B}$
and therefore its choice can have a significant impact on the overall efficiency of the hierarchy. For instance for unconstrained POPs, the Newton polytope method usually provides a monomial basis smaller than the standard monomial basis. However, this method} \newtssos{does not 
apply to constrained POPs. Here as an optional pre-treatment of POPs,
we provide an iterative procedure which not only enables us to obtain an initial  monomial basis $\mathscr{B}$ smaller than the one given by the Newton polytope method for unconstrained POPs in many cases, but can also be applied to constrained POPs. It sometimes leads to an initial monomial basis
$\mathscr{B}$ smaller than an obvious choice.}

We start with the unconstrained case. Let $f\in\R[\x]$ with $\A=\supp(f)$ and $\mathscr{B}$ the monomial basis given by the Newton polytope method. Set $\B_0:=\emptyset$. For $p\ge1$, we iteratively define a sequence of monomial sets $\{\B_p\}_p$ by
\begin{equation}
    \B_{p}:=\{\b\in\B\mid\exists\g\in\B\textrm{ s.t. }\b+\g\in\A\cup2\B_{p-1}\}.
\end{equation}

Consequently, we obtain an increasing chain of monomial sets:
$$\B_1\subseteq\B_2\subseteq\B_3\subseteq\cdots\subseteq\B.$$
Clearly, the above chain will stabilize in finite steps. Each $\B_p$ can serve as a candidate monomial basis. Particularly, we have
\begin{proposition}\label{sec5-prop1}
Let $f\in\R[\x]$ and $\B_*=\cup_{p\ge1}\B_p$. If $f\in\SDSOS$, then $f$ is an SDSOS polynomial in the monomial basis $\B_*$.
\end{proposition}
\begin{proof}
Let $\B$ be the monomial basis given by the Newton polytope method with $r=|\B|$. If $f\in\SDSOS$, then there exists a scaled diagonally dominant Gram matrix $Q\in\mathbb{S}_+^r$ indexed by $\B$ such that $f=(\x^{\mathscr{B}})^TQ\x^{\mathscr{B}}$. Let $s=|\B_*|$. We then construct a Gram matrix $\tilde{Q}\in\mathbf{S}_+^s$ indexed by $\B_*$ for $f$ as follows:
\begin{equation*}
\tilde{Q}_{\b\g}=
\begin{cases}
Q_{\b\g},\quad&\textrm{if } \b+\g\in\A\cup2\B_*, \\
0,\quad&\textrm{otherwise}.
\end{cases}
\end{equation*}
One can easily check that we still have $f=(\x^{\mathscr{B}_*})^T\tilde{Q}\x^{\mathscr{B}_*}$. Let $\hat{Q}$ be the principal submatrix of $Q$ by deleting the rows and columns whose indices are not in $\B_*$, which is also a scaled diagonally dominant matrix. By construction, $\tilde{Q}$ is obtained from $\hat{Q}$ by replacing certain off-diagonal entries by zeros.
Since replacing off-diagonal entries by zeros does not affect the scaled diagonal dominance of a matrix, $\tilde{Q}$ is also a scaled diagonally dominant matrix. It follows that $f$ is an SDSOS polynomial in the monomial basis $\B_*$.
\end{proof}

\begin{remark}
By Proposition \ref{sec5-prop1}, if we use the monomial basis $\B_*$ for $(\textrm{P}^k)$ \eqref{sec4-ukmom} and $(\textrm{P}^k)^*$ \eqref{sec4-uksos}, we still have the hierarchy of optimal values: $$\lambda^*\ge\lambda_{sos}\ge\cdots\ge\lambda_2\ge\lambda_1\ge\lambda_{sdsos}.$$
\end{remark}

\begin{algorithm}
\renewcommand{\algorithmicrequire}{\textbf{Input:}}
\renewcommand{\algorithmicensure}{\textbf{Output:}}
\caption{${\tt GenerateBasis}$}\label{alg1}
\begin{algorithmic}[1]
\Require
$\A$ and an initial monomial basis $\B$
\Ensure
An increasing chain of potential monomial bases $\{\B_p\}_{p\ge1}$
\State Set $\B_0:=\emptyset$;
\State Let $p=0$;
\While{$p=0$ or $\mathscr{B}_{p}\ne\mathscr{B}_{p-1}$}
\State $p:=p+1$
\State Set $\mathscr{B}_{p}:=\emptyset$;
\For{each pair $\{\b,\g\}$ of $\B$}
\If{$\b+\g\in\A\cup2\B_{p-1}$}
\State $\B_{p}:=\B_{p}\cup\{\b,\g\}$;
\EndIf
\EndFor
\EndWhile
\State \Return{$\{\B_p\}_{p\ge1}$};
\end{algorithmic}
\end{algorithm}

This method enables us to obtain a monomial basis which may be strictly smaller than the monomial basis given by the Newton polytope method as the following example shows.
\begin{example}
Consider the polynomial $f=1+x+x^8$. The monomial basis given by the Newton polytope method is $\B=\{1,x,x^2,x^3,x^4\}$. By the above iterative procedure, we compute that $\B_1=\{1,x,x^4\}$ and $\B_2=\{1,x,x^2,x^4\}$. It turns out that $\B_2$ can serve as a monomial basis to represent $f$ as an SOS.
\end{example}

\bigskip
For the constrained case \newtssos{we use notation of Sec.~\ref{sec4}}. Fix a relaxation order $\hat{d}$ and a sparse order $k$ of the \newtssos{chordal-TSSOS hierarchy}. Then each iteration breaks into two steps.

For Step 1, let the maximal cliques of $G_{j,\hat{d}}^{(k)}$ (\ref{sec4-graph}) be $C_{j,1}^{(k)},C_{j,2}^{(k)},\ldots,C_{j,l_j}^{(k)}$ for $j=0,\ldots,m$. Let
\begin{equation}\label{sec5-eq1}
\mathscr{F}=\supp(f)\cup\bigcup_{j=1}^m(\supp(g_j)+\bigcup_{i=1}^{l_j}(C_{j,i}^{(k)}+C_{j,i}^{(k)})).
\end{equation}
Then call the Algorithm {\tt GenerateBasis} with $\A=\mathscr{F}$ and $\B=\B_{0,\hat{d}}$ to generate a new monomial basis $\B_{0,\hat{d}}'$.

For Step 2, with the new monomial basis $\B_{0,\hat{d}}'$, we compute a new sparsity pattern graph $(G_{j,\hat{d}}^{(k)})'$ for $j=0,\ldots,m$. Then go back to Step 1.

Continue the iterative procedure until $\B_{0,\hat{d}}'=\B_{0,\hat{d}}$.

\section{Computational cost discussion}\label{sec6}
In this section, \newtssos{we provide an estimate for the computational cost associated with the chordal-TSSOS relaxation of POPs with sparse order $k=1$.} For \newtssos{ease of exposition} we only consider the unconstrained case \eqref{sec4-uksos} and use the standard monomial basis $\B=\N^n_d$.

Let $f\in\R[\x]$ be of degree $2d$ and $G_0$ the tsp graph associated with $f$. Let $\{G_k\}_{k\ge1}$ be defined by \eqref{sec3-graph}. By Theorem \ref{sec2-thm}, the complexity of \eqref{sec4-uksos} depends on two factors: the size of maximal cliques of $G_k$ and the number of equality constraints, i.e., $|\supp(G_k)|$. Since we rely on an approximately minimum chordal extension, only a small number of edges are added to $G_0$ in the process of obtaining $G_1$ from the chordal extension of $G_0$. In this case, the complexity of \eqref{sec4-uksos} for $k=1$ depends mainly on the size of maximal cliques of $G_0$ and $|\supp(G_0)|$.

For any $\a=(\alpha_i)\in\N^n$, we call $\a\,(\textrm{mod }2)=(\alpha_i\,(\textrm{mod }2))\in\{0,1\}^n$ the {\em sign type} of $\a$. We say that $\a$ is {\em even} if the sign type of $\a$ is $\mathbf{0}$ and is {\em odd} otherwise.
\begin{proposition}\label{sec6-prop1}
Let $G_0(V,E_0)$ be the tsp graph associated with $f$. Then for any $\b,\g\in V$ with the same sign type, $\{\b,\g\}\in E_0$.
\end{proposition}
\begin{proof}
It is immediate from the definition.
\end{proof}
Proposition \ref{sec6-prop1} implies that the nodes of $G_0$ with the same sign type form a clique of $G_0$.

\begin{proposition}\label{sec6-prop2}
Let $C$ be a subset of $\N^n_d$ such that the elements of $C$ have the same sign type. Then $|C|\le\binom{n+\lfloor \frac{d}{2}\rfloor}{\lfloor \frac{d}{2}\rfloor}$.
\end{proposition}
\begin{proof}
Let $\y=(y_1,\ldots,y_n)$ be a set of variables and $\s$ be the sign type of the elements in $C$. It follows that any $\a\in C$ is a nonnegative integer solution of the following system:
\begin{equation}\label{comp-eq1}
    \begin{cases}
    y_1+y_2+\cdots+y_n\le d,\\
    \y\,(\textrm{mod }2)=\s.
    \end{cases}
\end{equation}
Define $\bar{\y}=(\bar{y}_1,\cdots,\bar{y}_n)$ by
\begin{equation*}
    \bar{y}_i=\begin{cases}
    \frac{y_i}{2},\quad&\textrm{if } y_i\,(\textrm{mod }2)=0,\\
    \frac{y_i-1}{2},\quad&\textrm{if } y_i\,(\textrm{mod }2)=1.
    \end{cases}
\end{equation*}
Let $o$ be the number of subscripts $i$ such that $s_i=1$.
Then the nonnegative integer solution of \eqref{comp-eq1} is in a one-to-one correspondence to the nonnegative integer solution of the following system:
\begin{equation}\label{comp-eq2}
    \bar{y}_1+\bar{y}_2+\cdots+\bar{y}_n\le \lfloor \frac{d-o}{2}\rfloor.
\end{equation}
Since the number of nonnegative integer solutions to \eqref{comp-eq2} is $\binom{n+\lfloor \frac{d-o}{2}\rfloor}{\lfloor \frac{d-o}{2}\rfloor}$, we conclude that the number of nonnegative integer solutions to \eqref{comp-eq1} is also $\binom{n+\lfloor \frac{d-o}{2}\rfloor}{\lfloor \frac{d-o}{2}\rfloor}$. Therefore, we have $|C|\le\binom{n+\lfloor \frac{d-o}{2}\rfloor}{\lfloor \frac{d-o}{2}\rfloor}\le\binom{n+\lfloor \frac{d}{2}\rfloor}{\lfloor \frac{d}{2}\rfloor}$ as $o$ is a nonnegative integer.
\end{proof}

Combining Proposition \ref{sec6-prop1} and Proposition \ref{sec6-prop2}, we conclude that the size of any clique of $G_0$ whose nodes have the same sign type is no more than $\binom{n+\lfloor \frac{d}{2}\rfloor}{\lfloor \frac{d}{2}\rfloor}$.

Suppose $G(V,E)$ is a graph with $V=\B$. We say that an edge $\{\b,\g\}\in E$ is {\em even} if $\b+\g$ is even and is {\em odd} if $\b+\g$ is odd. In other words, an even edge connects two nodes of the same sign type and an odd edge connects two nodes of different sign types. From the definition of $G_0$, odd edges of $G_0$ correspond to odd vectors in $\supp(f)$. If $f$ is sufficiently sparse such that the odd edges of $G_0$ are not too many, then because of the above discussion, the maximal size of maximal cliques of $G_0$ is close to $\binom{n+\lfloor \frac{d}{2}\rfloor}{\lfloor \frac{d}{2}\rfloor}$.

It is known that for a chordal graph, the number of maximal cliques is less than the number of nodes \cite{ga}. Therefore, the number of maximal cliques of $G_0$ is bounded by $\binom{n+d}{d}$.

By construction, we have $|\supp(G_0)|=|\supp(f)\cup2\B|\le|\supp(f)|+\binom{n+d}{d}$.

\begin{example}
Consider the polynomial $f=1+x_1^4+x_2^4+x_3^4-x_1^2x_2^2-x_1^2x_3^2-x_2^2x_3^2+x_2x_3$.
See Figure \ref{complex} for the tsp graph $G_0$ of $f$. There are $6$ maximal cliques for $G_0$, which are of size $4,2,2,1,1,1$ respectively.
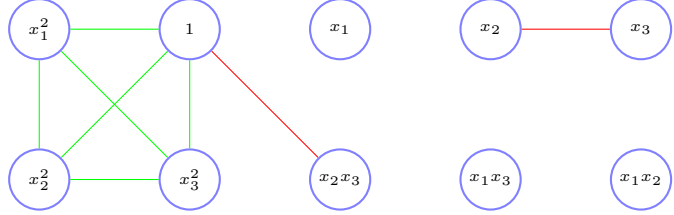
\begin{figure}[htbp]
\caption{The tsp graph $G_0$ of $f$}\label{complex}
\begin{center}
{\tiny
\begin{tikzpicture}[every node/.style={circle, draw=blue!50, thick, minimum size=8mm}]
\node (n1) at (0,0) {$1$};
\node (n8) at (-2,0) {$x_1^2$};
\node (n9) at (-2,-2) {$x_2^2$};
\node (n10) at (0,-2) {$x_3^2$};
\node (n2) at (2,0) {$x_1$};
\node (n3) at (4,0) {$x_2$};
\node (n4) at (6,0) {$x_3$};
\node (n5) at (2,-2) {$x_2x_3$};
\node (n6) at (4,-2) {$x_1x_3$};
\node (n7) at (6,-2) {$x_1x_2$};
\draw[green] (n1)--(n8);
\draw[green] (n1)--(n9);
\draw[green] (n1)--(n10);
\draw[green] (n8)--(n9);
\draw[green] (n8)--(n10);
\draw[green] (n9)--(n10);
\draw[red] (n1)--(n5);
\draw[red] (n3)--(n4);
\end{tikzpicture}}\\
{\small The green edges are even and the red edges are odd. Each node has a self-loop which is not displayed.}
\end{center}
\end{figure} 
\end{example}

On the other hand, for the dense SDP relaxation \eqref{sec2-usos} of \eqref{sec4-uksos}, there is only one SDP matrix which is of size $\binom{n+d}{d}$ and the number of equality constraints is $\binom{n+2d}{2d}$.

Thus we have the following table for the computational cost of the sparse (with the sparse order $k=1$) and dense SDP relaxations of \eqref{sec4-uksos}.
\begin{table}[htbp]
\caption{Computational cost comparison for the sparse and dense SDP relaxations of unconstrained POPs}
\begin{center}
\begin{tabular}{|c|c|c|c|}
\hline
&maximal size of SDP matrices&\#SDP matrices &\#equality constraints\\
\hline
Sparse&$\sim\binom{n+\lfloor \frac{d}{2}\rfloor}{\lfloor \frac{d}{2}\rfloor}$&$\le\binom{n+d}{d}$&$\sim|\supp(f)|+\binom{n+d}{d}$\\
\hline
Dense&$\binom{n+d}{d}$&$1$&$\binom{n+2d}{2d}$\\
\hline
\end{tabular}
\end{center}
\end{table}

We illustrate the above discussion by an explicit example.
\begin{example}
For $n\ge1$, let
\begin{equation}
f_{n}=\sum_{i=1}^n(x_i^2+x_i^4)+\sum_{i=1}^n\sum_{k=1}^n(x_i-x_k)^4.
\end{equation}
The tsp graph $G_0$ for $f_n$ (see Figure \ref{cost}) has $1$ maximal clique of size $n+1$ (involving the nodes $1,x_1^2,\ldots,x_n^2$), $\frac{n(n-1)}{2}$ maximal cliques of size $3$ (involving the nodes $x_i^2,x_j^2,x_ix_j$ for each pair $\{i,j\},i\ne j$) and $n$ maximal cliques of size $1$ (involving the node $x_i$ for each $i$). Note that $G_0$ is already a chordal graph. So we have $G_1=G_0$.
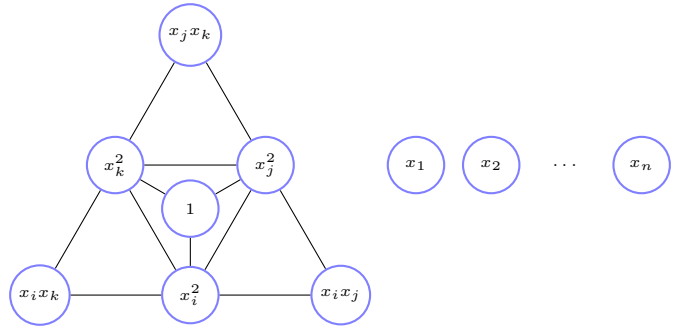
\begin{figure}[htbp]
\caption{The tsp graph $G_0$ of $f_n$}\label{cost}
\begin{center}
{\tiny
\begin{tikzpicture}[main node/.style={circle, draw=blue!50, thick, minimum size=7.5mm}]
\node[main node] (n1) at (1,0) {$x_j^2$};
\node[main node] (n2) at (-1,0) {$x_k^2$};
\node[main node] (n3) at (0,-1.73) {$x_i^2$};
\node[main node] (n10) at (0,-0.58) {$1$};
\draw (n10)--(n2);
\draw (n10)--(n3);
\draw (n10)--(n1);
\draw (n1)--(n2);
\draw (n1)--(n3);
\draw (n2)--(n3);
\node[main node] (n4) at (2,-1.73) {$x_ix_j$};
\draw (n1)--(n4);
\draw (n3)--(n4);
\node[main node] (n5) at (-2,-1.73) {$x_ix_k$};
\draw (n2)--(n5);
\draw (n3)--(n5);
\node[main node] (n6) at (0,1.73) {$x_jx_k$};
\draw (n1)--(n6);
\draw (n2)--(n6);
\node[main node] (n7) at (3,0) {$x_1$};
\node[main node] (n8) at (4,0) {$x_2$};
\node[main node] (n9) at (6,0) {$x_n$};
\path (n8) -- node[auto=false]{$\cdots$} (n9);
\end{tikzpicture}}\\
{\small This is a subgraph of $G_0$. The whole graph $G_0$ is obtained by putting all such subgraphs together. Each node has a self-loop which is not displayed.}
\end{center}
\end{figure} 

The computational cost for the sparse (with the sparse order $k=1$) and dense SDP relaxations of $f_{n}$ with $n$ is displayed in Table \ref{cc}. In the column `\#clique', $i\times j$ means $j$ cliques of size $i$.
\begin{table}[htbp]
\caption{Computational cost comparison for the sparse and dense SDP relaxations of $f_{n}$}\label{cc}
\begin{center}
\begin{tabular}{|c|c|c|}
\hline
&\#clique&\#equality constraint\\
\hline
Sparse&$3\times\frac{n(n-1)}{2},1\times n,(n+1)\times 1$&$\frac{3n(n-1)}{2}+2n+1$\\
\hline
Dense&$\binom{n+2}{2}\times 1$&$\binom{n+4}{4}$\\
\hline
\end{tabular}
\end{center}
\end{table}
\end{example}

\section{Numerical experiments}
\label{sec:benchs}
In this section, we present numerical results of the proposed \newtssos{chordal-TSSOS} hierarchies \eqref{sec4-ukmom}-\eqref{sec4-uksos} and \eqref{sec-eq1}-\eqref{sec6-eq1} for both unconstrained and constrained POPs, respectively.
Our tool, named TSSOS, is implemented in Julia, and constructs instances of the dual SDP problems \eqref{sec4-uksos} and \eqref{sec6-eq1}, then relies on MOSEK to solve them. TSSOS also implements the block moment-SOS hierarchy developed in \cite{wang2}. 
In the following subsections, we compare the performance of TSSOS with the one of GloptiPoly \cite{he} to solve the primal moment problem \eqref{sec2-eq11} of the dense hierarchy, and Yalmip \cite{lo} to solve its dual SOS problem \eqref{sec2-eq12}.
As for TSSOS, GloptiPoly and Yalmip rely on MOSEK to solve SDP problems.

Our TSSOS tool is available on the website:

\begin{center}
\url{https://github.com/wangjie212/TSSOS}.
\end{center}

All numerical examples were computed on an Intel Core i5-8265U@1.60GHz CPU with 8GB RAM memory and the WINDOWS 10 system.

The notation that we use are listed in Table 1.
\begin{table*}[htbp]\label{table1}
\caption{The notation}
\begin{center}
\begin{tabular}{|c|c|}
\hline
$n$&the number of variables\\
\hline
$2d$&the degree\\
\hline
$s$&the number of terms\\
\hline
$\hat{d}$&the relaxation order of Lasserre's hierarchy\\
\hline
$k$&the sparse order of the (block-) chordal-TSSOS hierarchy\\
\hline
bs&the size of initial monomial bases\\
\hline
rbs&the size of reduced monomial bases\\
\hline
\multirow{2}*{mc}&whose $k$-th entry is the maximal size of maximal cliques (blocks)\\
&obtained from the (block-) chordal-TSSOS hierarchy at sparse order $k$\\
\hline
\multirow{2}*{opt}&whose $k$-th entry is the optimum\\
&obtained from the (block-) chordal-TSSOS hierarchy at sparse order $k$\\
\hline
\multirow{2}*{time}&whose $k$-th entry is the time (in seconds) of computing\\
&the solution of the (block-) chordal-TSSOS hierarchy at sparse order $k$\\
\hline
$0$&a number whose absolute value is less than $1e\textrm{-}5$\\
\hline
-&out of memory\\
\hline
\end{tabular}
\end{center}
\end{table*}

\subsection{Unconstrained polynomial optimization problems}
We first present the numerical results for randomly generated polynomials of two types. The first type is of the SOS form. More concretely, we consider the polynomial $$f=\sum_{i=1}^tf_i^2\in\textbf{randpoly1}(n,2d,t,p) \,,$$ constructed as follows: first randomly choose a subset of monomials $M$ from $\x^{\N^n_{d}}$ with probability $p$, and then randomly assign the elements of $M$ to $f_1,\ldots,f_t$ with random coefficients between $-1$ and $1$. We generate $18$ random polynomials $F_1,\ldots,F_{18}$ from $6$ different classes, where $$F_1,F_2,F_3\in\textbf{randpoly1}(8,8,30,0.1),$$ $$F_4,F_5,F_6\in\textbf{randpoly1}(8,10,25,0.04),$$
$$F_{7},F_{8},F_{9}\in\textbf{randpoly1}(9,10,30,0.03),$$
$$F_{10},F_{11},F_{12}\in\textbf{randpoly1}(10,12,20,0.01),$$
$$F_{13},F_{14},F_{15}\in\textbf{randpoly1}(10,16,30,0.003),$$
$$F_{16},F_{17},F_{18}\in\textbf{randpoly1}(12,12,50,0.01).$$
For these polynomials, we first compute a monomial basis using the Newton polytope method (\ref{sec2-eq3}) and then reduce the monomial basis using the methods introduced in Sec.~\ref{sec4}.
Table \ref{tb:randpoly1} displays the numerical results on these polynomials. Note that the time spent to compute a monomial basis is included in the running time to solve the SDP \eqref{sec4-uksos} for $k=1$.
This explains why the total running time is less important at $k=2$ for $F_2$, $F_3$ and $F_4$.

In Table \ref{tb:randpoly1c}, we compare the performance of TSSOS, GloptiPoly and Yalmip on these polynomials. We present the performance of TSSOS for both ``\newtssos{chordal}'' and ``block'' approaches.
In Yalmip, we turn the option ``sos.newton'' on to compute a monomial basis by the Newton polytope method. For TSSOS, we only present the results of the first three chordal-TSSOS hierarchy since for all instances except $F_3$, the sequence of graphs $\{G_k\}_{k\ge1}$ stabilizes in three steps.

As the tables illustrate, TSSOS is significantly faster than GloptiPoly and Yalmip. In addition, TSSOS scales much better than GloptiPoly and Yalmip. GloptiPoly can only handle polynomials with $n=8$, $d=8$ and Yalmip can only handle polynomials with $n\le10$, $d\le12$ due to the memory constraint, while TSSOS can easily handle all of these polynomials. Moreover, one can see that the chordal-TSSOS hierarchy performs much better than the block-TSSOS hierarchy.

\begin{table}[htbp]
\caption{The results for randomly generated polynomials of type \uppercase\expandafter{\romannumeral1}}\label{tb:randpoly1}
\begin{center}
\begin{tabular}{|c|c|c|c|c|c|c|c|c|}
\hline
&$n$&$2d$&$s$&bs&rbs&mc&opt&time\\
\hline
$F_1$&$8$&$8$&$64$&$106$&$49$&$4$&$0$&$0.24$\\
\hline
$F_2$&$8$&$8$&$102$&$122$&$76$&$6,6$&$0,0$&$0.34,0.08$\\
\hline
$F_3$&$8$&$8$&$104$&$150$&$95$&$6,8,11$&$0,0,0$&$0.36,0.05,0.08$\\
\hline
$F_4$&$8$&$10$&$103$&$202$&$75$&$5,5$&$0,0$&$0.58,0.04$\\
\hline
$F_5$&$8$&$10$&$85$&$201$&$70$&$4$&$0$&$0.53$\\
\hline
$F_6$&$8$&$10$&$111$&$128$&$60$&$7$&$0$&$0.38$\\
\hline
$F_{7}$&$9$&$10$&$101$&$145$&$66$&$5$&$0$&$0.50$\\
\hline
$F_{8}$&$9$&$10$&$166$&$178$&$81$&$5$&$0$&$0.72$\\
\hline
$F_{9}$&$9$&$10$&$161$&$171$&$78$&$8$&$0$&$0.79$\\
\hline
$F_{10}$&$10$&$12$&$271$&$223$&$94$&$8$&$0$&$2.2$\\
\hline
$F_{11}$&$10$&$12$&$253$&$176$&$88$&$9$&$0$&$1.6$\\
\hline
$F_{12}$&$10$&$12$&$261$&$204$&$94$&$8$&$0$&$1.8$\\
\hline
$F_{13}$&$10$&$16$&$370$&$1098$&$125$&$9$&$0$&$15$\\
\hline
$F_{14}$&$10$&$16$&$412$&$800$&$139$&$9$&$0$&$14$\\
\hline
$F_{15}$&$10$&$16$&$436$&$618$&$146$&$9$&$0$&$12$\\
\hline
$F_{16}$&$12$&$12$&$488$&$330$&$187$&$8$&$0$&$8.4$\\
\hline
$F_{17}$&$12$&$12$&$351$&$264$&$150$&$8$&$0$&$5.7$\\
 \hline
$F_{18}$&$12$&$12$&$464$&$316$&$179$&$8$&$0$&$7.4$\\
\hline
\end{tabular}
\end{center}
\end{table}

\begin{table}[htbp]
\caption{Comparison with GloptiPoly and Yalmip for randomly generated polynomials of type \uppercase\expandafter{\romannumeral1}}\label{tb:randpoly1c}
\begin{center}
\begin{tabular}{|c|c|c|c|c||c|c|c|c|c|}
\hline
&\multicolumn{4}{c||}{time}&&\multicolumn{4}{c|}{time}\\
\cline{2-5}\cline{7-10}
&\multicolumn{2}{c|}{TSSOS}&\multirow{2}*{GloptiPoly}&\multirow{2}*{Yalmip}&&\multicolumn{2}{c|}{TSSOS}&\multirow{2}*{GloptiPoly}&\multirow{2}*{Yalmip}\\
\cline{2-3}\cline{7-8}
& \newtssos{chordal} &block&&&&\newtssos{chordal}&block&&\\
\hline
$F_1$&$0.24$&$1.7$&$306$&$10$&$F_{10}$&$2.2$&$12$&-&$474$\\
\hline
$F_2$&$0.34$&$4.6$&$348$&$13$&$F_{11}$&$1.6$&$9.2$&-&$147$\\
\hline
$F_3$&$0.36$&$8.8$&$326$&$19$&$F_{12}$&$1.8$&$12$&-&$350$\\
\hline
$F_4$&$0.58$&$4.8$&-&$92$&$F_{13}$&$15$&$36$&-&-\\
\hline
$F_5$&$0.53$&$4.2$&-&$72$&$F_{14}$&$14$&$305$&-&-\\
\hline
$F_6$&$0.38$&$5.2$&-&$22$&$F_{15}$&$12$&$207$&-&-\\
\hline
$F_7$&$0.50$&$3.2$&-&$44$&$F_{16}$&$8.4$&$61$&-&-\\
\hline
$F_8$&$0.72$&$6.5$&-&$143$&$F_{17}$&$5.7$&$17$&-&-\\
\hline
$F_9$&$0.79$&$5.9$&-&$109$&$F_{18}$&$7.4$&$22$&-&-\\
\hline
\end{tabular}
\end{center}
\end{table}

The second type of randomly generated problems are polynomials whose Newton polytopes are scaled standard simplices.
More concretely, we consider polynomials defined by $$f=c_0+\sum_{i=1}^nc_ix_i^{2d}+\sum_{j=1}^{s-n-1}c_j'\x^{\a_j}\in\textbf{randpoly2}(n,2d,s) \,,$$ constructed as follows: we randomly choose coefficients $c_i$ between $0$ and $1$, as well as $s-n-1$ vectors $\a_j$ in $\N^n_{2d-1}\backslash\{\mathbf{0}\}$ with random coefficients $c_j'$ between $-1$ and $1$. We generate $18$ random polynomials $G_1,\ldots,G_{18}$ from $6$ different classes, where $$G_1,G_2,G_3\in\textbf{randpoly2}(8,8,15),$$ $$G_4,G_5,G_6\in\textbf{randpoly2}(9,8,20),$$
$$G_7,G_8,G_9\in\textbf{randpoly2}(9,10,15),$$
$$G_{10},G_{11},G_{12}\in\textbf{randpoly2}(10,8,20),$$
$$G_{13},G_{14},G_{15}\in\textbf{randpoly2}(11,8,20),$$
$$G_{16},G_{17},G_{18}\in\textbf{randpoly2}(12,8,25).$$
Table \ref{tb:randpoly2} displays the numerical results on these polynomials. 
It happens that the second (or the third) sparse hierarchy spends less time than the first one because it involves less maximal cliques than the first one while the sizes of maximal cliques are close.

In Table \ref{tb:randpoly2c}, we compare the performance of TSSOS, GloptiPoly and Yalmip on these polynomials. We present the performance of TSSOS for both ``\newtssos{chordal}'' and ``block'' \cite{wang2} approaches.
In Yalmip, we turn the option ``sos.congruence'' on to take sign-symmetries into account, which allows one to handle slightly more polynomials than GloptiPoly. For TSSOS, we only present the results of the first three \newtssos{chordal-TSSOS} hierarchy since it always converges to the same optimum with the dense moment-SOS relaxation in three steps. 

Again as the tables illustrate, TSSOS is significantly faster than GloptiPoly and Yalmip. Also, TSSOS scales much better than GloptiPoly and Yalmip. GloptiPoly can only handle polynomials with $n=8$, $d=8$ and Yalmip can only handle a portion of polynomials with $n\le11$, $d\le10$ due to the memory constraint, while TSSOS can easily handle all of these polynomials. Moreover, one can see that the chordal-TSSOS hierarchy performs much better than the block-TSSOS hierarchy.

\begin{table}[htbp]
\caption{The results for randomly generated polynomials of type \uppercase\expandafter{\romannumeral2}}\label{tb:randpoly2}
\begin{center}
\begin{tabular}{|c|c|c|c|c|c|c|c|c|}
\hline
&$n$&$2d$&$s$&bs&rbs&mc&opt&time\\
\hline
$G_1$&$8$&$8$&$15$&$495$&$235$&$21,23,28$&$-0.5758,-0.5758,-0.5758$&$0.36,0.21,0.28$\\
\hline
$G_2$&$8$&$8$&$15$&$495$&$328$&$31,33,37$&$-34.69,-34.69,-34.69$&$0.51,0.57,1.7$\\
\hline
$G_3$&$8$&$8$&$15$&$495$&$258$&$21,23,31$&$0.7073,0.7073,0.7073$&$0.31,0.23,0.36$\\
\hline
$G_4$&$9$&$8$&$20$&$715$&$415$&$31,41,127$&$-801.7,-801.7,-801.7$&$1.0,1.8,184$\\
\hline
$G_5$&$9$&$8$&$20$&$715$&$342$&$28,31,40$&$-0.8064,-0.8064,-0.8064$&$0.63,0.45,1.1$\\
\hline
$G_6$&$9$&$8$&$20$&$715$&$340$&$25,43,61$&$-1.698,-1.698,-1.698$&$0.76,1.4,3.6$\\
\hline
$G_{7}$&$9$&$10$&$15$&$2002$&$1254$&$38,41,55$&$-1.295,-1.295,-1.295$&$6.6,5.2,15$\\
\hline
$G_{8}$&$9$&$10$&$15$&$2002$&$894$&$26,28,40$&$-0.6622,-0.6622,-0.6622$&$5.0,1.8,3.5$\\
\hline
$G_{9}$&$9$&$10$&$15$&$2002$&$888$&$28,31,31$&$0.5180,0.5180,0.5180$&$4.9,1.4,2.3$\\
\hline
$G_{10}$&$10$&$8$&$20$&$1001$&$454$&$31,36,42$&$-0.4895,-0.4895,-0.4895$&$1.2,0.73,0.92$\\
\hline
$G_{11}$&$10$&$8$&$20$&$1001$&$414$&$26,33,60$&$0.1732,0.1798,0.1867$&$1.1,0.87,6.0$\\
\hline
$G_{12}$&$10$&$8$&$20$&$1001$&$387$&$23,37,52$&$0.4943,0.4943,0.4943$&$1.0,1.2,2.4$\\
\hline
$G_{13}$&$11$&$8$&$20$&$1365$&$299$&$21,22,22$&$-3.963,-3.963,-3.963$&$1.7,0.26,0.30$\\
\hline
$G_{14}$&$11$&$8$&$20$&$1365$&$412$&$27,33,42$&$-2.184,-2.184,-2.184$&$1.8,0.68,3.3$\\
\hline
$G_{15}$&$11$&$8$&$20$&$1365$&$458$&$27,30,37$&$0.0588,0.0588,0.0588$&$1.9,0.59,0.76$\\
\hline
$G_{16}$&$12$&$8$&$25$&$1820$&$744$&$39,58,81$&$-758.6,-688.0,-688.0$&$4.1,5.8,73$\\
\hline
$G_{17}$&$12$&$8$&$25$&$1820$&$694$&$37,51,76$&$-40.89,-40.22,-40.22$&$3.7,3.7,31$\\
\hline
$G_{18}$&$12$&$8$&$25$&$1820$&$581$&$31,40,48$&$-14.27,-14.27,-14.27$&$2.9,1.3,1.8$\\
\hline
\end{tabular}
\end{center}
\end{table}

\begin{table}[htbp]
\caption{Comparison with GloptiPoly and Yalmip for randomly generated polynomials of type \uppercase\expandafter{\romannumeral2}}\label{tb:randpoly2c}
\begin{center}
\begin{tabular}{|c|c|c|c|c||c|c|c|c|c|}
\hline
&\multicolumn{4}{c||}{time}&&\multicolumn{4}{c|}{time}\\
\cline{2-5}\cline{7-10}
&\multicolumn{2}{c|}{TSSOS}&\multirow{2}*{GloptiPoly}&\multirow{2}*{Yalmip}&&\multicolumn{2}{c|}{TSSOS}&\multirow{2}*{GloptiPoly}&\multirow{2}*{Yalmip}\\
\cline{2-3}\cline{7-8}
& \newtssos{chordal} &block&&&& \newtssos{chordal} &block&&\\
\hline
$G_1$&$0.36$&$8.5$&$346$&$31$&$G_{10}$&$1.2$&$13$&-&-\\
\hline
$G_2$&$0.51$&$2.6$&$447$&$24$&$G_{11}$&$8.0$&$86$&-&$536$\\
\hline
$G_3$&$0.31$&$1.0$&$257$&$6.0$&$G_{12}$&$1.0$&$66$&-&-\\
\hline
$G_4$&$1.0$&$40$&-&-&$G_{13}$&$1.7$&$13$&-&$655$\\
\hline
$G_5$&$0.63$&$24$&-&$363$&$G_{14}$&$1.8$&$37$&-&-\\
\hline
$G_6$&$0.76$&$31$&-&$141$&$G_{15}$&$1.9$&$36$&-&$340$\\
\hline
$G_7$&$6.6$&$24$&-&$322$&$G_{16}$&$10$&$693$&-&-\\
\hline
$G_8$&$5.0$&$28$&-&$233$&$G_{17}$&$7.4$&$333$&-&-\\
\hline
$G_9$&$4.9$&$21$&-&$249$&$G_{18}$&$2.9$&$393$&-&-\\
\hline
\end{tabular}
\end{center}
\end{table}

\bigskip
The Broyden banded function (\cite{waki}) is defined by
\begin{equation*}
    f_{\textrm{Bb}}(\x)=\sum_{i=1}^n(x_i(2+5x_i^2)+1-\sum_{j\in J_i}(1+x_j)x_j)^2,
\end{equation*}
where $J_i=\{j\mid j\ne i, \max(1,i-5)\le j\le\min(n,i+1)\}$.
Table \ref{tb:bdfunction} displays the results of Broyden banded functions for the chordal-TSSOS hierarchy and SparsePOP \cite{WakiKKMS08} (SparsePOP uses SeDuMi as an SDP solver). The optimum is always $0$. So we only provide the data of running time. For this example, since TSSOS and SparsePOP use different SDP solvers, the running time is not comparable directly. We thereby also provide the number of SDP variables involved in TSSOS and SparsePOP respectively. This example shows that even though both tsp and csp apply to a POP, it still happens that the tsp leads to a smaller size of cliques than the csp and thus saves more computational effort for the corresponding SDPs.

\begin{table}[htbp]
\caption{The results for Broyden banded functions}\label{tb:bdfunction}
\begin{center}
\begin{tabular}{|c|c|c|c|c|c|c|}
\hline
\multicolumn{2}{|c|}{$n$}&$6$&$7$&$8$&$9$&$10$\\
\hline
\multirow{2}*{mc}&TSSOS&$15$&$18$&$19$&$22$&$22$\\
\cline{2-7}
&SparsePOP&$84$&$120$&$120$&$120$&$120$\\
\hline
\multirow{2}*{\#SDP variables}&TSSOS&$2792$&$3588$&$4555$&$5247$&$6697$\\
\cline{2-7}
&SparsePOP&$7056$&$14400$&$28800$&$43200$&$57600$\\
\hline
\multirow{2}*{time}&TSSOS&$0.10$&$0.12$&$0.15$&$0.18$&$0.25$\\
\cline{2-7}
&SparsePOP&$2.0$&$9.0$&$20$&$30$&$42$\\
\hline
\end{tabular}
\end{center}
\end{table}

\bigskip
In the final part of this section, we present the numerical results for the following two functions:

$\bullet$ The modified generalized Rosenbrock function
\begin{equation*}
    f_{\textrm{mgR}}(\x)=1+\sum_{i=1}^n(100(x_i-x_{i-1}^2)^2+(1-x_i)^2)+\sum_{i=1}^n\sum_{j=i+1}^nx_i^2x_j^2,
\end{equation*}
which is obtained from the generalized Rosenbrock function by adding monomials terms leading to a complete csp graph, thus yielding dense SDP relaxations.

$\bullet$ The modified chained singular function
\begin{align*}
    f_{\textrm{mcs}}(\x)=&\sum_{i\in J}((x_{i}+10x_{i+1})^2+5(x_{i+2}-x_{i+3})^2+(x_{i+1}-2x_{i+2})^4\\&+10(x_{i}-10x_{i+3})^4+\sum_{i=1}^n\sum_{j=i+1}^nx_i^2x_j^2,
\end{align*}
with $J=\{1,3,5,\ldots,n-3\}$ which is obtained from the chained singular function by adding the same monomial terms as above.

The results for modified generalized Rosenbrock and chained singular functions are displayed in Table \ref{tb:mgrfunction} and  Table \ref{tb:mcsfunction}, respectively. As the tables illustrate, the \newtssos{chordal-TSSOS} hierarchy can handle these functions with variables up to $200$ while the block-TSSOS hierarchy can handle these functions with variables no more than $100$ due to the memory constraint. 

\begin{table}[htbp]
\caption{The results for modified generalized Rosenbrock functions}\label{tb:mgrfunction}
\begin{center}
\begin{tabular}{|c|c|c|c|c|c|c|}
\hline
\multirow{3}*{$n$}&\multicolumn{6}{c|}{TSSOS}\\
\cline{2-7}
&\multicolumn{3}{c|}{\newtssos{chordal}}&\multicolumn{3}{c|}{block}\\
\cline{2-7}
&mc&opt&time&mc&opt&time\\
\hline
$10$&$11$&$8.45$&$0.03$&$28,56$&$8.45,8.45$&$0.06,0.32$\\
\hline
$20$&$21$&$18.35$&$0.12$&$58,211$&$18.35,18.35$&$1.2,50$\\
\hline
$30$&$31$&$28.25$&$0.34$&$88,466$&$28.25,-$&$9,-$\\
\hline
$40$&$41$&$38.15$&$0.87$&$118,821$&$38.15,-$&$42,-$\\
\hline
$50$&$51$&$48.05$&$2.5$&$148,1276$&$48.05,-$&$146,-$\\
\hline
$60$&$61$&$57.95$&$4.1$&$178,1831$&$57.95,-$&$382,-$\\
\hline
$70$&$71$&$67.85$&$9.5$&$218,2486$&$67.85,-$&$786,-$\\
\hline
$80$&$81$&$77.75$&$22$&$248,3241$&$77.75,-$&$4467,-$\\
\hline
$90$&$91$&$87.65$&$28$&$278,4096$&$-,-$&$-,-$\\
\hline
$100$&$101$&$97.55$&$46$&$308,5051$&$-,-$&$-,-$\\
\hline
$120$&$121$&$117.35$&$105$&$368,7261$&$-,-$&$-,-$\\
\hline
$140$&$141$&$137.15$&$243$&$428,9871$&$-,-$&$-,-$\\
\hline
$160$&$161$&$156.95$&$504$&$488,12881$&$-,-$&$-,-$\\
\hline
$180$&$181$&$176.75$&$820$&$548,16291$&$-,-$&$-,-$\\
\hline
$200$&$201$&$196.55$&$1792$&$608,20101$&$-,-$&$-,-$\\
\hline
\end{tabular}
\end{center}
\end{table}

\begin{table}[htbp]
\caption{The results for modified chained singular functions}\label{tb:mcsfunction}
\begin{center}
\begin{tabular}{|c|c|c|c|c|c|c|}
\hline
\multirow{3}*{$n$}&\multicolumn{6}{c|}{TSSOS}\\
\cline{2-7}
&\multicolumn{3}{c|}{\newtssos{chordal}}&\multicolumn{3}{c|}{block}\\
\cline{2-7}
&mc&opt&time&mc&opt&time\\
\hline
$10$&$11$&$-0.0003$&$0.06$&$24,56$&$-0.0006,-0.0007$&$0.07,0.42$\\
\hline
$20$&$21$&$-0.0013$&$0.11$&$49,211$&$-0.0006,-0.0007$&$0.77,78$\\
\hline
$30$&$31$&$-0.0004$&$0.37$&$74,466$&$-0.0002,-$&$3.9,-$\\
\hline
$40$&$41$&$-0.0007$&$0.85$&$99,821$&$-0.0001,-$&$15,-$\\
\hline
$50$&$51$&$-0.0021$&$2.1$&$124,1276$&$-0.0006,-$&$45,-$\\
\hline
$60$&$61$&$-0.0021$&$4.7$&$149,1831$&$-0.0002,-$&$112,-$\\
\hline
$70$&$71$&$-0.0030$&$7.6$&$174,2486$&$-0.0005,-$&$282,-$\\
\hline
$80$&$81$&$-0.0040$&$19$&$199,3241$&$-0.0002,-$&$670,-$\\
\hline
$90$&$91$&$-0.0034$&$23$&$224,4096$&$-0.0004,-$&$1768,-$\\
\hline
$100$&$101$&$-0.0038$&$37$&$249,5051$&$-,-$&$-,-$\\
\hline
$120$&$121$&$-0.0014$&$88$&$274,7261$&$-,-$&$-,-$\\
\hline
$140$&$141$&$-0.0011$&$199$&$299,9871$&$-,-$&$-,-$\\
\hline
$160$&$161$&$-0.0015$&$390$&$324,12881$&$-,-$&$-,-$\\
\hline
$180$&$181$&$-0.0057$&$682$&$349,16291$&$-,-$&$-,-$\\
\hline
$200$&$201$&$-0.0083$&$1126$&$374,20101$&$-,-$&$-,-$\\
\hline
\end{tabular}
\end{center}
\end{table}

\subsection{Constrained polynomial optimization problems}
Now we present the numerical results for constrained polynomial optimization problems. 
We first consider six randomly generated polynomials $H_1,\ldots,H_6$ of type \uppercase\expandafter{\romannumeral2} as objective functions $f$ and minimize them over the two following semialgebraic sets: the unit ball $$\mathbf{K}=\{(x_1,\ldots,x_n)\in\R^n\mid g_1=1-(x_1^2+\cdots+x_n^2)\ge0\} \,,$$ and the unit hypercube $$\mathbf{K}=\{(x_1,\ldots,x_n)\in\R^n\mid g_1=1-x_1^2\ge0,\ldots,g_n=1-x_n^2\ge0\}.$$

The results for the unit ball case are displayed in Table \ref{unitball} and the results for the unit hypercube case are displayed in Table \ref{unithypercube} (only the results of the first three chordal-TSSOS hierarchy are displayed). We compare the performance of the chordal-TSSOS hierarchy with that of GloptiPoly. It can be seen that for each instance TSSOS is significantly faster than GloptiPoly without compromising accuracy.

\begin{table}[htbp]
\caption{The results for minimizing randomly generated polynomials of type \uppercase\expandafter{\romannumeral2} over unit balls}\label{unitball}
\begin{center}
\begin{tabular}{|c|c|c|c|c|c|c|c|c|c|}
\hline
&\multirow{2}*{$(n,2d,s)$}&\multirow{2}*{$\hat{d}$}&\multirow{2}*{$k$}&\multirow{2}*{mc}&\multicolumn{2}{c|}{TSSOS}&\multicolumn{2}{c|}{GloptiPoly}\\
\cline{6-9}
&&&&&opt&time&opt&time\\
\hline
\multirow{6}*{$H_1$}&$\multirow{6}*{(6,8,10)}$&\multirow{3}*{$4$}&$1$&$(28,7)$&\multirow{6}*{$0.1362$}&$0.20$&\multirow{6}*{$0.1362$}&\multirow{3}*{$8.0$}\\
\cline{4-5}\cline{7-7}
&&&$2$&$(32,12)$& &$0.52$&&\\
\cline{4-5}\cline{7-7}
&&&$3$&$(37,20)$& &$0.86$&&\\
\cline{3-5} \cline{7-7} \cline{9-9}
&&\multirow{3}*{$5$}&$1$&$(29,28)$& &$0.91$& &\multirow{3}*{$80$}\\
\cline{4-5}\cline{7-7}
&&&$2$&$(35,30)$& &$3.0$&&\\
\cline{4-5}\cline{7-7}
&&&$3$&$(48,45)$& &$9.0$&&\\
\hline
\multirow{6}*{$H_2$}&$\multirow{6}*{(7,8,12)}$&\multirow{3}*{$4$}&$1$&$(36,8)$&\multirow{6}*{$0.1373$}&$0.36$&\multirow{3}*{$0.1373$}&\multirow{3}*{$34$}\\
\cline{4-5}\cline{7-7}
&&&$2$&$(36,10)$&&$0.52$&&\\
\cline{4-5}\cline{7-7}
&&&$3$&$(38,15)$&&$1.6$&&\\
\cline{3-5} \cline{7-9}
&&\multirow{3}*{$5$}&$1$&$(36,36)$&&$1.9$&\multirow{3}*{-}&\multirow{3}*{-}\\
\cline{4-5}\cline{7-7}
&&&$2$&$(45,36)$&&$3.9$&&\\
\cline{4-5}\cline{7-7}
&&&$3$&$(59,49)$&&$34$&&\\
\hline
\multirow{6}*{$H_3$}&$\multirow{6}*{(8,8,15)}$&\multirow{3}*{$4$}&$1$&$(45,9)$&\multirow{6}*{$0.1212$}&$0.75$&\multirow{3}*{$0.1212$}&\multirow{3}*{$225$}\\
\cline{4-5}\cline{7-7}
&&&$2$&$(45,10)$&&$1.3$&&\\
\cline{4-5}\cline{7-7}
&&&$3$&$(53,25)$&&$20$&&\\
\cline{3-5} \cline{7-9}
&&\multirow{3}*{$5$}&$1$&$(45,45)$&&$5.3$&\multirow{3}*{-}&\multirow{3}*{-}\\
\cline{4-5}\cline{7-7}
&&&$2$&$(45,45)$&&$7.5$&&\\
\cline{4-5}\cline{7-7}
&&&$3$&$(59,46)$&&$94$&&\\
\hline
\multirow{6}*{$H_4$}&$\multirow{6}*{(9,6,15)}$&\multirow{3}*{$3$}&$1$&$(10,10)$&\multirow{6}*{$0.8704$}&$0.15$&\multirow{3}*{$0.8704$}&\multirow{3}*{$16$}\\
\cline{4-5}\cline{7-7}
&&&$2$&$(10,10)$&&$0.22$&&\\
\cline{4-5}\cline{7-7}
&&&$3$&$(10,10)$&&$0.25$&&\\
\cline{3-5} \cline{7-9}
&&\multirow{3}*{$4$}&$1$&$(55,10)$&&$1.3$&\multirow{3}*{-}&\multirow{3}*{-}\\
\cline{4-5}\cline{7-7}
&&&$2$&$(55,13)$&&$2.0$&&\\
\cline{4-5}\cline{7-7}
&&&$3$&$(56,19)$&&$2.8$&&\\
\hline
\multirow{6}*{$H_5$}&$\multirow{6}*{(10,6,20)}$&\multirow{3}*{$3$}&$1$&$(12,11)$&\multirow{6}*{$0.5966$}&$0.22$&\multirow{3}*{$0.5966$}&\multirow{3}*{$48$}\\
\cline{4-5}\cline{7-7}
&&&$2$&$(13,14)$&&$0.42$&&\\
\cline{4-5}\cline{7-7}
&&&$3$&$(19,16)$&&$0.95$&&\\
\cline{3-5} \cline{7-9}
&&\multirow{3}*{$4$}&$1$&$(66,13)$&&$2.5$&\multirow{3}*{-}&\multirow{3}*{-}\\
\cline{4-5}\cline{7-7}
&&&$2$&$(66,23)$&&$10$&&\\
\cline{4-5}\cline{7-7}
&&&$3$&$(75,44)$&&$88$&&\\
\hline
\multirow{6}*{$H_6$}&$\multirow{6}*{(11,6,20)}$&\multirow{3}*{$3$}&$1$&$(12,12)$&\multirow{6}*{$0.1171$}&$0.28$&\multirow{3}*{$0.1171$}&\multirow{3}*{$115$}\\
\cline{4-5}\cline{7-7}
&&&$2$&$(15,12)$&&$0.36$&&\\
\cline{4-5}\cline{7-7}
&&&$3$&$(16,13)$&&$0.60$&&\\
\cline{3-5} \cline{7-9}
&&\multirow{3}*{$4$}&$1$&$(78,14)$&&$4.4$&\multirow{3}*{-}&\multirow{3}*{-}\\
\cline{4-5}\cline{7-7}
&&&$2$&$(78,15)$&&$4.7$&&\\
\cline{4-5}\cline{7-7}
&&&$3$&$(78,13)$&&$7.5$&&\\
\hline
\end{tabular}\\
{\small In this table, the first entry of mc is the maximal size of maximal cliques corresponding to the moment matrix $M_{\hat{d}}(\y)$ and the second entry of mc is the maximal size of maximal cliques corresponding to the localizing matrix $M_{\hat{d}-d_1}(g_1\y)$.}
\end{center}
\end{table}

\begin{table}[htbp]
\caption{The results for minimizing randomly generated polynomials of type \uppercase\expandafter{\romannumeral2} over unit hypercubes}\label{unithypercube}
\begin{center}
\begin{tabular}{|c|c|c|c|c|c|c|c|c|c|}
\hline
&\multirow{2}*{$(n,2d,s)$}&\multirow{2}*{$\hat{d}$}&\multirow{2}*{$k$}&\multirow{2}*{mc}&\multicolumn{2}{c|}{TSSOS}&\multicolumn{2}{c|}{GloptiPoly}\\
\cline{6-9}
&&&&&opt&time&opt&time\\
\hline
\multirow{6}*{$H_1$}&$\multirow{6}*{(6,8,10)}$&\multirow{3}*{$4$}&$1$&$(28,8)$&\multirow{6}*{$-0.4400$}&$0.33$&\multirow{6}*{$-0.4400$}&\multirow{3}*{$19$}\\
\cline{4-5}\cline{7-7}
&&&$2$&$(32,12)$&&$0.86$&&\\
\cline{4-5}\cline{7-7}
&&&$3$&$(37,19)$&&$1.4$&&\\
\cline{3-5} \cline{7-7} \cline{9-9}
&&\multirow{3}*{$5$}&$1$&$(29,28)$&&$1.7$&&\multirow{3}*{$237$}\\
\cline{4-5}\cline{7-7}
&&&$2$&$(35,30)$&&$6.1$&&\\
\cline{4-5}\cline{7-7}
&&&$3$&$(48,45)$&&$16$&&\\
\hline
\multirow{6}*{$H_2$}&$\multirow{6}*{(7,8,12)}$&\multirow{3}*{$4$}&$1$&$(36,9)$&\multirow{6}*{$-0.1289$}&$0.80$&\multirow{3}*{$-0.1289$}&\multirow{3}*{$101$}\\
\cline{4-5}\cline{7-7}
&&&$2$&$(36,10)$&&$0.94$&&\\
\cline{4-5}\cline{7-7}
&&&$3$&$(38,13)$&&$1.9$&&\\
\cline{3-5} \cline{7-9}
&&\multirow{3}*{$5$}&$1$&$(36,36)$&&$3.6$&\multirow{3}*{-}&\multirow{3}*{-}\\
\cline{4-5}\cline{7-7}
&&&$2$&$(45,36)$&&$6.8$&&\\
\cline{4-5}\cline{7-7}
&&&$3$&$(59,49)$&&$61$&&\\
\hline
\multirow{6}*{$H_3$}&$\multirow{6}*{(8,8,15)}$&\multirow{3}*{$4$}&$1$&$(45,10)$&\multirow{6}*{$-0.1465$}&$1.0$&\multirow{3}*{$-0.1465$}&\multirow{3}*{$433$}\\
\cline{4-5}\cline{7-7}
&&&$2$&$(45,11)$&&$2.0$&&\\
\cline{4-5}\cline{7-7}
&&&$3$&$(53,22)$&&$24$&&\\
\cline{3-5} \cline{7-9}
&&\multirow{3}*{$5$}&$1$&$(45,45)$&&$9.6$&\multirow{3}*{-}&\multirow{3}*{-}\\
\cline{4-5}\cline{7-7}
&&&$2$&$(45,45)$&&$13$&&\\
\cline{4-5}\cline{7-7}
&&&$3$&$(59,50)$&&$112$&&\\
\hline
\multirow{6}*{$H_4$}&$\multirow{6}*{(9,6,15)}$&\multirow{3}*{$3$}&$1$&$(10,10)$&\multirow{6}*{$0.1199$}&$0.27$&\multirow{3}*{$0.1199$}&\multirow{3}*{$27$}\\
\cline{4-5}\cline{7-7}
&&&$2$&$(10,10)$&&$0.30$&&\\
\cline{4-5}\cline{7-7}
&&&$3$&$(10,10)$&&$0.32$&&\\
\cline{3-5} \cline{7-9}
&&\multirow{3}*{$4$}&$1$&$(55,11)$&&$2.1$&\multirow{3}*{-}&\multirow{3}*{-}\\
\cline{4-5}\cline{7-7}
&&&$2$&$(55,13)$&&$2.8$&&\\
\cline{4-5}\cline{7-7}
&&&$3$&$(56,19)$&&$4.0$&&\\
\hline
\multirow{6}*{$H_5$}&$\multirow{6}*{(10,6,20)}$&\multirow{3}*{$3$}&$1$&$(12,11)$&\multirow{6}*{$-0.2813$}&$0.38$&\multirow{3}*{$-0.2813$}&\multirow{3}*{$69$}\\
\cline{4-5}\cline{7-7}
&&&$2$&$(13,12)$&&$0.50$&&\\
\cline{4-5}\cline{7-7}
&&&$3$&$(19,13)$&&$0.70$&&\\
\cline{3-5} \cline{7-9}
&&\multirow{3}*{$4$}&$1$&$(66,14)$&&$4.4$&\multirow{3}*{-}&\multirow{3}*{-}\\
\cline{4-5}\cline{7-7}
&&&$2$&$(66,23)$&&$15$&&\\
\cline{4-5}\cline{7-7}
&&&$3$&$(75,44)$&&$80$&&\\
\hline
\multirow{6}*{$H_6$}&$\multirow{6}*{(11,6,20)}$&\multirow{3}*{$3$}&$1$&$(12,12)$&\multirow{6}*{$-0.2316$}&$0.47$&\multirow{3}*{$-0.2316$}&\multirow{3}*{$211$}\\
\cline{4-5}\cline{7-7}
&&&$2$&$(15,12)$&&$0.61$&&\\
\cline{4-5}\cline{7-7}
&&&$3$&$(16,13)$&&$0.76$&&\\
\cline{3-5} \cline{7-9}
&&\multirow{3}*{$4$}&$1$&$(78,13)$&&$7.5$&\multirow{3}*{-}&\multirow{3}*{-}\\
\cline{4-5}\cline{7-7}
&&&$2$&$(78,15)$&&$9.9$&&\\
\cline{4-5}\cline{7-7}
&&&$3$&$(78,13)$&&$13$&&\\
\hline
\end{tabular}\\
{\small In this table, the first entry of mc is the maximal size of maximal cliques corresponding to the moment matrix $M_{\hat{d}}(\y)$ and the second entry of mc is the maximal size of maxiaml cliques corresponding to the localizing matrices $M_{\hat{d}-d_j}(g_j\y),j=1,\ldots,n$.}
\end{center}
\end{table}

\bigskip
Next we present the numerical results of the following two functions over the unit ball.

$\bullet$ The Broyden tridiagonal function
\begin{align*}
    f_{\textrm{Bt}}(\x)=&((3-2x_1)x_1-2x_2+1)^2+\sum_{i=1}^{n-1}((3-2x_i)x_i-x_{i-1}-2x_{i+1}+1)^2\\&+((3-2x_n)x_n-x_{n-1}+1)^2.
\end{align*}

$\bullet$ The generalized Rosenbrock function
\begin{equation*}
    f_{\textrm{gR}}(\x)=1+\sum_{i=1}^n(100(x_i-x_{i-1}^2)^2+(1-x_i)^2).
\end{equation*}

Since the constraint of unit balls involves all of the variables, the csp graph for these problems is clearly complete.
The results for generalized Rosenbrock functions are displayed in Table \ref{tb:gRfunction} and the results for Broyden tridiagonal functions are displayed in Table \ref{tb:mBtfunction}. The relaxation order is $2$. We present the performance of TSSOS for both ``chordal'' and ``block'' approaches. As the tables illustrate, again the \newtssos{chordal-TSSOS} hierarchy performs much better than the block-TSSOS hierarchy. 

\begin{table}[htbp]
\caption{The results for generalized Rosenbrock functions}\label{tb:gRfunction}
\begin{center}
\begin{tabular}{|c|c|c|c|c|c|c|c|c|}
\hline
\multirow{3}*{$n$}&\multicolumn{8}{c|}{TSSOS}\\
\cline{2-9}
&\multicolumn{4}{c|}{\newtssos{chordal}}&\multicolumn{4}{c|}{block}\\
\cline{2-9}
&$k$&mc&opt&time&$k$&mc&opt&time\\
\hline
\multirow{2}*{$10$}&\multirow{2}*{$1$}&\multirow{2}*{$(11,2)$}&\multirow{2}*{$8.35$}&\multirow{2}*{$0.05$}&$1$&$(28,10)$&$8.35$&$0.22$\\
\cline{6-9}
&&&&&$2$&$(56,10)$&$8.35$&$0.32$\\
\hline
\multirow{2}*{$20$}&\multirow{2}*{$1$}&\multirow{2}*{$(21,2)$}&\multirow{2}*{$18.25$}&\multirow{2}*{$0.19$}&$1$&$(58,20)$&$18.25$&$8.2$\\
\cline{6-9}
&&&&&$2$&$(211,20)$&$18.25$&$45$\\
\hline
\multirow{2}*{$30$}&\multirow{2}*{$1$}&\multirow{2}*{$(31,2)$}&\multirow{2}*{$28.15$}&\multirow{2}*{$0.49$}&$1$&$(88,30)$&$28.15$&$203$\\
\cline{6-9}
&&&&&$2$&$(466,30)$&-&-\\
\hline
\multirow{2}*{$40$}&\multirow{2}*{$1$}&\multirow{2}*{$(41,2)$}&\multirow{2}*{$38.05$}&\multirow{2}*{$1.3$}&$1$&$(118,40)$&-&-\\
\cline{6-9}
&&&&&$2$&$(821,40)$&-&-\\
\hline
\multirow{2}*{$50$}&\multirow{2}*{$1$}&\multirow{2}*{$(51,2)$}&\multirow{2}*{$47.95$}&\multirow{2}*{$4.0$}&$1$&$(148,50)$&-&-\\
\cline{6-9}
&&&&&$2$&$(1276,50)$&-&-\\
\hline
\multirow{2}*{$60$}&\multirow{2}*{$1$}&\multirow{2}*{$(61,2)$}&\multirow{2}*{$57.85$}&\multirow{2}*{$6.6$}&$1$&$(178,60)$&-&-\\
\cline{6-9}
&&&&&$2$&$(1831,60)$&-&-\\
\hline
\multirow{2}*{$70$}&\multirow{2}*{$1$}&\multirow{2}*{$(71,2)$}&\multirow{2}*{$67.75$}&\multirow{2}*{$18$}&$1$&$(218,70)$&-&-\\
\cline{6-9}
&&&&&$2$&$(2486,70)$&-&-\\
\hline
\multirow{2}*{$80$}&\multirow{2}*{$1$}&\multirow{2}*{$(81,2)$}&\multirow{2}*{$77.65$}&\multirow{2}*{$26$}&$1$&$(248,78)$&-&-\\
\cline{6-9}
&&&&&$2$&$3241,80)$&-&-\\
\hline
\multirow{2}*{$90$}&\multirow{2}*{$1$}&\multirow{2}*{$(91,2)$}&\multirow{2}*{$87.55$}&\multirow{2}*{$50$}&$1$&$(278,90)$&-&-\\
\cline{6-9}
&&&&&$2$&$(4096,90)$&-&-\\
\hline
\multirow{2}*{$100$}&\multirow{2}*{$1$}&\multirow{2}*{$(101,2)$}&\multirow{2}*{$97.45$}&\multirow{2}*{$85$}&$1$&$(308,100)$&-&-\\
\cline{6-9}
&&&&&$2$&$(5051,100)$&-&-\\
\hline
\multirow{2}*{$120$}&\multirow{2}*{$1$}&\multirow{2}*{$(121,2)$}&\multirow{2}*{$117.25$}&\multirow{2}*{$186$}&$1$&$(368,120)$&-&-\\
\cline{6-9}
&&&&&$2$&$(7261,120)$&-&-\\
\hline
\multirow{2}*{$140$}&\multirow{2}*{$1$}&\multirow{2}*{$(141,2)$}&\multirow{2}*{$137.05$}&\multirow{2}*{$448$}&$1$&$(428,140)$&-&-\\
\cline{6-9}
&&&&&$2$&$(9871,140)$&-&-\\
\hline
\multirow{2}*{$160$}&\multirow{2}*{$1$}&\multirow{2}*{$(161,2)$}&\multirow{2}*{$156.85$}&\multirow{2}*{$841$}&$1$&$(488,160)$&-&-\\
\cline{6-9}
&&&&&$2$&$(12881,160)$&-&-\\
\hline
\multirow{2}*{$180$}&\multirow{2}*{$1$}&\multirow{2}*{$(181,2)$}&\multirow{2}*{$176.65$}&\multirow{2}*{$1495$}&$1$&$(548,180)$&-&-\\
\cline{6-9}
&&&&&$2$&$(16291,180)$&-&-\\
\hline
\end{tabular}\\
{\small In this table, the first entry of mc is the maximal size of maximal cliques corresponding to the moment matrix $M_{\hat{d}}(\y)$ and the second entry of mc is the maximal size of maximal cliques corresponding to the localizing matrix $M_{\hat{d}-d_1}(g_1\y)$.}
\end{center}
\end{table}

\begin{table}[htbp]
\caption{The results for Broyden tridiagonal functions}\label{tb:mBtfunction}
\begin{center}
\begin{tabular}{|c|c|c|c|c|c|c|c|c|}
\hline
\multirow{3}*{$n$}&\multicolumn{8}{c|}{TSSOS}\\
\cline{2-9}
&\multicolumn{4}{c|}{\newtssos{chordal}}&\multicolumn{4}{c|}{block}\\
\cline{2-9}
&$k$&mc&opt&time&$k$&mc&opt&time\\
\hline
\multirow{2}*{$10$}&$1$&$(13,5)$&$5.15$&$0.07$&$1$&$(38,11)$&$5.15$&$0.20$\\
\cline{2-9}
&$2$&$(13,5)$&$5.15$&$0.08$&$2$&$(66,11)$&$5.15$&$0.32$\\
\hline
\multirow{2}*{$20$}&$1$&$(23,5)$&$15.04$&$0.37$&$1$&$(78,21)$&$15.04$&$11$\\
\cline{2-9}
&$2$&$(23,7)$&$15.04$&$0.50$&$2$&$(231,21)$&$15.04$&$73$\\
\hline
\multirow{2}*{$30$}&$1$&$(33,5)$&$25.01$&$1.2$&$1$&$(118,31)$&$25.01$&$174$\\
\cline{2-9}
&$2$&$(33,7)$&$25.01$&$2.1$&$2$&$(496,31)$&$25.01$&$-$\\
\hline
\multirow{2}*{$40$}&$1$&$(43,5)$&$35.00$&$3.6$&$1$&$(158,41)$&$-$&$-$\\
\cline{2-9}
&$2$&$(43,7)$&$35.00$&$7.8$&$2$&$(861,41)$&$-$&$-$\\
\hline
\multirow{2}*{$50$}&$1$&$(53,5)$&$44.99$&$9.0$&$1$&$(198,51)$&$-$&$-$\\
\cline{2-9}
&$2$&$(53,7)$&$44.99$&$21$&$2$&$(1326,51)$&$-$&$-$\\
\hline
\multirow{2}*{$60$}&$1$&$(65,5)$&$54.99$&$18$&$1$&$(238,61)$&$-$&$-$\\
\cline{2-9}
&$2$&$(65,7)$&$54.99$&$48$&$2$&$(1891,61)$&$-$&$-$\\
\hline
\multirow{2}*{$70$}&$1$&$(73,5)$&$64.99$&$41$&$1$&$(278,71)$&$-$&$-$\\
\cline{2-9}
&$2$&$(73,7)$&$64.99$&$138$&$2$&$(2556,71)$&$-$&$-$\\
\hline
\multirow{2}*{$80$}&$1$&$(83,5)$&$74.99$&$67$&$1$&$(318,81)$&$-$&$-$\\
\cline{2-9}
&$2$&$(83,7)$&$74.99$&$240$&$2$&$(3321,81)$&$-$&$-$\\
\hline
\multirow{2}*{$90$}&$1$&$(93,5)$&$84.99$&$110$&$1$&$(348,91)$&$-$&$-$\\
\cline{2-9}
&$2$&$(93,7)$&$84.99$&$193$&$2$&$(4186,91)$&$-$&$-$\\
\hline
\multirow{2}*{$100$}&$1$&$(103,5)$&$94.98$&$188$&$1$&$(378,101)$&$-$&$-$\\
\cline{2-9}
&$2$&$(103,7)$&$94.98$&$299$&$2$&$(5151,101)$&$-$&$-$\\
\hline
\multirow{2}*{$120$}&$1$&$(123,4)$&$114.98$&$374$&$1$&$(458,121)$&$-$&$-$\\
\cline{2-9}
&$2$&$(123,8)$&$114.98$&$864$&$2$&$(7381,121)$&$-$&$-$\\
\hline
\end{tabular}\\
{\small In this table, the first entry of mc is the maximal size of maximal cliques corresponding to the moment matrix $M_{\hat{d}}(\y)$ and the second entry of mc is the maximal size of maximal cliques corresponding to the localizing matrix $M_{\hat{d}-d_1}(g_1\y)$.}
\end{center}
\end{table}

\section{Conclusions and outlooks}\label{cons}
In this paper, we continue to exploit the term-sparsity pattern (tsp) graph for a POP, \newtssos{as a follow-up of our previous work \cite{wang2}}. 
Through the support-extension and chordal-extension operations, we iteratively enlarge this tsp graph to obtain a \newtssos{chordal-TSSOS} two-level hierarchy for POPs. 
Various numerical examples demonstrate the efficiency and the scalability of this new hierarchy for unconstrained and constrained POPs.

There are still many questions leaving for further investigations:

1) When relying on the dense moment-SOS hierarchy, one can extract the global optimizers under certain flatness conditions of the moment matrix \cite{henrion2005detecting}. 
A similar procedure exists for the sparse moment-SOS hierarchy based on correlative sparsity \cite{Las06,ncsparse}.
It is worth looking for a similar condition when relying on our \newtssos{chordal-TSSOS} hierarchy.

2) We have the freedom to choose a specific chordal extension. The size of maximal cliques, the convergence and the convergence rate of the \newtssos{chordal-TSSOS} hierarchy highly depend on this choice. 
In \cite{wang2}, we have studied the block-TSSOS hierarchy which differs from the \newtssos{chordal-TSSOS} hierarchy by  selecting the chordal-extension operation instead of the completion of connected components. 
This choice leads to a maximal chordal-extension.
In this paper, we mainly focus on an  (approximately) minimal chordal extension. Usually, this extension provides maximal cliques of size smaller than the maximal chordal extension, at the cost of loosing theoretical convergence guarantees.
However, as \cite{sli} recently suggests, the  minimal chordal extension is not always optimal. 
To the best of our knowledge, people paid little attention to chordal extensions that are not (approximately) minimal. 
Therefore it is a future task to explore more general choices of chordal extensions as well as to look for an optimal one for specific POPs.

3) The \newtssos{chordal-TSSOS} hierarchy can be easily combined with other techniques for handling large-scale POPs, e.g., correlative sparsity  \cite{waki} or structured subsets \cite{mi}. We plan to do it in a follow-up paper.


\begin{thebibliography}{99}

\bibitem{ahmadi2014dsos}
A. A. Ahmadi and A. Majumdar, {\em DSOS and SDSOS optimization: LP and SOCP-based alternatives to sum of squares optimization}, 48th annual conference on information sciences and systems (CISS), 2014:1-5.

\bibitem{am}
P. R. Amestoy, T. A. Davis, I. S. Duff, {\em Algorithm 837: AMD, an approximate minimum degree ordering algorithm}, ACM Transactions on Mathematical Software, 30(3), 381-388 (2004).

\bibitem{be}
A. Berry, J. R. S. Blair, P. Heggernes, B. W. Peyton, {\em  Maximum cardinality search for computing minimal triangulations of graphs}, Algorithmica, 39(4), 287-298 (2004).

\bibitem{bp}
J. R. S. Blair, B. Peyton, {\em An introduction to chordal graphs and clique trees}, in Graph Theory and Sparse Matrix Computation, A. George, J. R. Gilbert, and J. W. H. Liu, eds.,
Springer-Verlag, New York, 1-29 (1993).

\bibitem{Chandrasekaran:Shah:SAGE}
V. Chandrasekaran and P. Shah, {\em Relative Entropy Relaxations for Signomial
  Optimization.}
  \newblock \emph{SIAM J. Optim.} 26(2):1147--1173, 2016.

\bibitem{re2}
M. D. Choi, T. Y. Lam, and B. Reznick, {\em Sums of squares of real polynomials}, Proceedings of Symposia in Pure mathematics, AMS, 58(1995): 103-126.

\bibitem{fukuda2001exploiting}
M. Fukuda, M. Kojima, K. Murota, and K. Nakata.
\newblock Exploiting sparsity in semidefinite programming via matrix
  completion. {I}. {G}eneral framework.
\newblock {\em SIAM J. Optim.}, 11(3):647--674, 2000/01.

\bibitem{fg}
D. R. Fulkerson, O. A. Gross, {\em Incidence matrices and interval graphs}, Pacific J. Math., 15, 835-855 (1965).

\bibitem{ga}
F. Gavril, {\em Algorithms for minimum coloring, maximum clique, minimum
covering by cliques, and maximum independent set of a chordal graph},
SIAM Journal on Computing, 1(2):180-187, 1972.

\bibitem{Ghasemi:Marshall:GP}
M. Ghasemi and M. Marshall, {\em Lower bounds for polynomials using geometric programming},
\newblock \emph{SIAM J. Optim.}, 22(2):460--473, 2012.

\bibitem{go}
M. C. Golumbic, {\em Algorithmic Graph Theory and Perfect Graphs}, Academic Press, New York (1980).

\bibitem{han}
E. J. Hancock and A. Papachristodoulou, {\em Structured Sum of Squares for Networked Systems Analysis}, 50th IEEE Conference on Decision and Control and European Control Conference (CDC-ECC).

\bibitem{heg}
P. Heggernes, {\em Minimal triangulations of graphs: a survey}, Discrete Mathematics, 306(3), 297-317 (2006).

\bibitem{he}
D. Henrion and J. B. Lasserre, {\em GloptiPoly: Global Optimization over Polynomials with Matlab and SeDuMi}, IEEE Conf. Decis. Control, Las Vegas, Nevada, 2002:747-752.

\bibitem{henrion2005detecting}
D. Henrion and J. B. Lasserre. 
\newblock  Detecting global optimality and extracting solutions in GloptiPoly, 
\newblock \emph{Positive polynomials in control}, 293--310, 2005.

\bibitem{Iliman:deWolff:Circuits}
S. Iliman and T. de Wolff, {\em Amoebas, nonnegative polynomials and sums of
  squares supported on circuits}, Res. Math. Sci. 3:3-9, 2016.

\bibitem{Josz16}
C\'{e}dric Josz.
\newblock {\em {Application of polynomial optimization to electricity
  transmission networks}}.
\newblock Theses, {Universit{\'e} Pierre et Marie Curie - Paris VI}, July 2016.

\bibitem{ncsparse}
I. Klep, V. Magron, J. Povh, {\em Sparse Noncommutative Polynomial Optimization}.
\newblock {\em preprint   \href{http://arxiv.org/abs/1909.00569}{arXiv:1909.00569}}, 2019.

\bibitem{ko}
M. Kojima, S. Kim, H. Waki, {\em Sparsity in sums of squares of polynomials}, Math. Program., 103(2005):45-62.

\bibitem{las1}
J. B. Lasserre, {\em Global optimization with polynomials and the problem of moments}, SIAM Journal on Optimization, 11(3)(2001):796-817.

\bibitem{Las06}
J.-B. Lasserre.
\newblock Convergent {SDP}-relaxations in polynomial optimization with
  sparsity.
\newblock {\em SIAM J. Optim.}, 17(3):822--843, 2006.

\bibitem{TohLasserre}
J.-B. Lasserre, K.-C. Toh, and S. Yang.
\newblock A bounded degree {SOS} hierarchy for polynomial optimization.
\newblock {\em EURO J. Comput. Optim.}, 5(1-2):87--117, 2017.

\bibitem{Lau09b}
M. Laurent.
\newblock Sums of squares, moment matrices and optimization over polynomials.
\newblock In {\em Emerging applications of algebraic geometry}, volume 149 of
  {\em IMA Vol. Math. Appl.}, pages 157--270. Springer, New York, 2009.

\bibitem{lo}
J. L\"ofberg, YALMIP: a toolbox for modeling and optimization in MATLAB, In 2004 IEEE International Conference on Robotics and Automation (IEEE Cat. No.04CH37508), 284-289.

\bibitem{lo1}
J. L\"ofberg, {\em Pre- and Post-Processing Sum-of-Squares Programs in Practice}, IEEE Transactions on Automatic Control, 54(5)(2009):1007-1011.

\bibitem{toms17}
V. Magron, G. Constantinides, and A. Donaldson.
\newblock Certified roundoff error bounds using semidefinite programming.
\newblock {\em ACM Trans. Math. Software}, 43(4):Art. 34, 31, 2017.

\bibitem{toms18}
V. Magron.
\newblock {Interval Enclosures of Upper Bounds of Roundoff Errors Using
  Semidefinite Programming}.
\newblock {\em ACM Trans. Math. Softw.}, 44(4):41:1--41:18, June 2018.

\bibitem{maj}
A. Majumdar, A. A. Ahmadi and R. Tedrake, {\em Control and verification of high-dimensional systems with DSOS and SDSOS programming}, In 53rd IEEE Conference on Decision and Control, 2014:394-401.

\bibitem{ma}
A. Marandi, E. D. Klerk, J. Dahl, {\em Solving sparse polynomial optimization problems with chordal structure using the sparse bounded-degree sum-of-squares hierarchy}, Discrete Applied Mathematics, 2017.

\bibitem{me}
A. Megretski, Systems polynomial optimization tools (SPOT), 2010. Available at https://github.com/spot-toolbox/spotless.

\bibitem{mi}
J. Miller, Y. Zheng, M. Sznaier, and A. Papachristodoulou, {\em Decomposed Structured Subsets for Semidefinite and Sum-of-Squares Optimization},
{\em preprint   \href{http://arxiv.org/abs/1911.12859}{arXiv:1911.12859}}, 2019.

\bibitem{nakata2003exploiting}
K. Nakata, K. Fujisawa, M. Fukuda, M. Kojima, and K.
  Murota.
\newblock Exploiting sparsity in semidefinite programming via matrix
  completion. {II}. {I}mplementation and numerical results.
\newblock {\em Math. Program.}, 95(2, Ser. B):303--327, 2003.

\bibitem{nie2014optimality}
J. Nie.
\newblock Optimality conditions and finite convergence of Lasserre’s hierarchy.
\newblock {\em Math. Program.}, 146(1-2, Ser. A):97--121, 2014.

\bibitem{pa}
P. A. Parrilo, {\em Structured semidefinite programs and semialgebraic geometry methods in robustness and optimization}, Ph.D. Thesis, California Institute of Technology, 2000.

\bibitem{pe}
F. Permenter, P. A. Parrilo, {\em Basis selection for SOS programs via facial reduction and polyhedral approximations}, Decision and Control, IEEE, 2014:6615-6620.

\bibitem{pe1}
F. Permenter, P. A. Parrilo, {\em Finding sparse, equivalent SDPs using minimal coordinate projections}, In 54th IEEE Conference on Decision and Control, CDC 2015, Osaka, Japan,
December 15-18, 2015:7274-7279.

\bibitem{pu}
M. Putinar, {\em Positive polynomials on compact semialgebraic sets}, Indiana Univ. Math. J., 42(1993):969-984.

\bibitem{re}
B. Reznick, {\em Extremal PSD forms with few terms}, Duke Math. J., 45(1978):363-374.

\bibitem{Riener13}
C. Riener, T. Theobald, L.~J. Andr\'{e}n, and J.-B. Lasserre.
\newblock Exploiting symmetries in {SDP}-relaxations for polynomial
  optimization.
\newblock {\em Math. Oper. Res.}, 38(1):122--141, 2013.

\bibitem{sli}
J. Sliwak, M. Anjos, L. L\'etocart, J. Maeght, and E. Traversi, {\em Improving Clique Decompositions of Semidefinite Relaxations for Optimal Power Flow Problems}, arXiv preprint arXiv:1912.09232, 2019.

\bibitem{Tacchi19}
M. Tacchi, T. Weisser, J.-B. Lasserre, and D. Henrion.
\newblock Exploiting sparsity for semi-algebraic set volume computation.
\newblock {\em preprint   \href{http://arxiv.org/abs/1902.02976}{arXiv:1902.02976}}, 2019.

\bibitem{va}
L. Vandenberghe, M. S. Andersen, {\em Chordal Graphs and Semidefinite Optimization}, Foundations and Trends in Optimization, 1(4), 241-433, Now Publisher (2015).

\bibitem{wang}
J. Wang, H. Li and B. Xia, {\em A New Sparse SOS Decomposition Algorithm Based on Term Sparsity}, in Proceedings of the 2019 on International Symposium on Symbolic and Algebraic Computation, ACM, 2019:347-354.

\bibitem{wang2}
J. Wang, V. Magron and J.-B. Lasserre, {\em TSSOS: A Moment-SOS hierarchy that exploits term sparsity}, arXiv preprint arXiv:1912.08899, 2019

\bibitem{waki}
H. Waki, S. Kim, M. Kojima, and M. Muramatsu, {\em Sums of squares and semidefinite program relaxations for polynomial optimization problems with structured sparsity}, SIAM Journal on Optimization, 17(1)(2016):218-242.

\bibitem{WakiKKMS08}
H. Waki, S. Kim, M. Kojima, M. Muramatsu, and H. Sugimoto.
\newblock Algorithm 883: sparse{POP}---a sparse semidefinite programming relaxation of polynomial optimization problems.
\newblock {\em ACM Trans. Math. Software}, 35(2):Art. 15, 13, 2009.

\bibitem{we}
T. Weisser, J. B. Lasserre, and K. C. Toh, {\em Sparse-BSOS: a bounded degree SOS hierarchy for large scale polynomial optimization with sparsity}, Mathematical Programming Computation, 10(1)(2018):1-32.

\end{thebibliography}
\end{document}